\documentclass[reqno,11pt]{amsart}
\usepackage{amsmath, amsfonts,amsthm,amssymb,amsbsy,upref,color,graphicx,amscd,hyperref,enumerate}
\usepackage{color}
\usepackage[active]{srcltx}
\usepackage[latin1]{inputenc}
\usepackage{bbm}
\usepackage{graphicx}
\usepackage{subfig} 
\usepackage{xfrac}
\usepackage{comment} 

\usepackage{dsfont}

\def\ds{\displaystyle}
\def\eps{{\varepsilon}}
\def\N{\mathbb{N}}

\def\R{\mathbb{R}}
\def\mS{\mathbb{S}}

\def\F{\mathcal{F}}
\def\HH{\mathcal{H}}

\def\M{\mathcal{M}}

\newcommand{\be}{\begin{equation}}
\newcommand{\ee}{\end{equation}}

\newcommand{\de}{\partial}

\newcommand{\ind}{\mathbbm{1}}

\def\cF{\mathcal{F}}

\def\G{\mathcal{G}}

\theoremstyle{plain}
\newtheorem{teo}{Theorem}[section]
\newtheorem{lm}[teo]{Lemma}
\newtheorem{prop}[teo]{Proposition}
\newtheorem{coro}[teo]{Corollary}

\theoremstyle{definition}
\newtheorem{definition}[teo]{Definition}

\newtheorem{oss}[teo]{Remark}
\newtheorem{exam}[teo]{Example}

\def\XXint#1#2#3{{\setbox0=\hbox{$#1{#2#3}{\int}$}
     \vcenter{\hbox{$#2#3$}}\kern-.5\wd0}}

\usepackage{refcount}

\newcounter{cte}

\numberwithin{equation}{section}

\begin{document}

\title[parabolic variational inequalities]{On the asymptotic behavior of the solutions to parabolic variational inequalities}
\author{Maria Colombo, Luca Spolaor, Bozhidar Velichkov}

\address {Maria Colombo: \newline \indent
Institute for Theoretical Studies, ETH Z\"urich,
	\newline \indent
 Clausiusstrasse 47, CH-8092 Z\"urich, Switzerland
 	}
\email{maria.colombo@eth-its.ethz.ch}

\address {Luca Spolaor: \newline \indent
	Massachusetts Institute of Technology (MIT), 
	\newline \indent
	77 Massachusetts Avenue, Cambridge 
	MA 02139, USA}
\email{lspolaor@mit.edu}

\address {Bozhidar Velichkov: \newline \indent
Universit\'e Grenoble Alpes
\newline \indent
B\^atiment IMAG, 700 Avenue Centrale, 38401 Saint-Martin-d'H\`eres, France}
\email{bozhidar.velichkov@univ-grenoble-alpes.fr}

\keywords{parabolic variational inequalities, \L ojasiewicz inequality, obstacle problem, thin-obstacle problem, epiperimetric inequality}
\subjclass{35R35, 35K85, 35B35, 35B40, 35B41}
%
\date{\today}
\maketitle

\begin{abstract} We consider various versions of the obstacle and thin-obstacle problems, we interpret them as variational inequalities, with non-smooth constraint, and prove that they satisfy a new \emph{constrained \L ojasiewicz inequality}. The difficulty lies in the fact that, since the constraint is non-analytic, the pioneering method of L.\,Simon (\cite{Simon0}) does not apply and we have to exploit a better understanding on the constraint itself. We then apply this inequality to two associated problems. First we combine it with an abstract result on parabolic variational inequalities, to prove the convergence at infinity of the strong global solutions to the parabolic obstacle and thin-obstacle problems to a unique stationary solution with a rate. 
Secondly, we give an abstract proof, based on a parabolic approach, of the epiperimetric inequality, which we then apply to the singular points of the obstacle and thin-obstacle problems.
\end{abstract}



\section{Introduction} 

In this paper we consider parabolic variational inequalities of the form
\begin{equation}\label{e:intro:para}
\begin{cases}
\big(u'(t)+\nabla \cF (u(t))\big)\cdot (v-u(t))\ge0\,,\quad\text{for every}\quad v\in \mathcal K\quad\text{and}\quad t>0\,,\\
u(0)=u_0\in \mathcal K\,,
\end{cases} 
\end{equation}
where $\mathcal F$ is a given analytic integral functional, $\mathcal K$ is a convex subset of $L^2(\Omega)$ and the dot stands for the scalar product in $L^2(\Omega)$; $\Omega$ being a smooth domain in $\R^d$ or a $d$-dimensional manifold. We provide a new method for the study of the asymptotic behavior of the solution at infinity and we apply it to the parabolic obstacle and thin-obstacle problems, which are related to several relevant physical models 
(for more details and an extensive reference list we refer to the books \cite{Friedman, DuLi}).
\smallskip

In the absence of the constraint $\mathcal K$, the parabolic problem \eqref{e:intro:para} reduces to the infinite dimensional gradient flow of $\mathcal F$, which is given by
$$u'(t)=-\nabla\mathcal F(u(t))\quad\text{for every}\quad t>0,\qquad u(0)=u_0.$$
In this case, it is well known (for more details we refer to Subsection \ref{sub:intro:loja1}) that the asymptotic behavior of the solution can be deduced by the so-called \L ojasiewicz inequality, that is  for every stationary point $\varphi$ of $\mathcal F$ there are constants $\gamma\in\,]0,\sfrac12]$, $C>0$ such that
\begin{equation}\label{e:intro:loja}
\big(\mathcal F(u)-\mathcal F(\varphi)\big)_+^{1-\gamma}\le C\|\nabla \mathcal F(u)\|_{L^2}\quad\text{for every}\ u\ \text{in a neighborhood of}\ \varphi\,.
\end{equation}
Precisely,  \eqref{e:intro:loja} implies that: 
\begin{equation}\label{e:intro:claim}\,
\!\!\!\!\begin{array}{cc}
\!\!\!\!\text{\it There is a neighborhood $\mathcal U$ of $\varphi$ such that: if $u_0\in\mathcal U$ and $\mathcal F(u(t))\ge\mathcal F(\varphi)$, for every $t\ge 0$,}\\
\!\!\!\!\text{\it then $u(t)$ converges, as $t\to\infty$, to a critical point $u_\infty$ of $\mathcal F$, with a rate depending on $\gamma$.} 
\end{array}
\end{equation}
In the seminal paper \cite{Simon0}, Leon Simon proved \eqref{e:intro:claim} for the flow associated to harmonic maps between two analytic manifolds. Notice that, also in this case, 
there is a geometric constraint given by the target manifold, but a change of coordinates allows to trivialize this constraint, while transforming the Dirichlet energy into an analytic functional $\mathcal F$. In \cite{Simon0}, Simon showed that the analyticity of this functional 
allows to reduce \eqref{e:intro:loja} to the classical \L ojasiewicz inequality \eqref{e:loja0} for analytic functions in $\R^n$.
\medskip

Suppose now that $\mathcal K$ is a non-analytic convex set of $L^2$, as in the case of the obstacle and the thin-obstacle problems. From one side, the non-smooth nature of $\mathcal K$ does not allow the use of a Lagrange multiplier argument
in order to replace the constraint by an additional analytic term in the energy $\mathcal F$. On the other hand, just the analyticity of $\mathcal F$, and the consequent \eqref{e:intro:loja}, are not sufficient to obtain the convergence result \eqref{e:intro:claim} (see Example \ref{exam:non_analytic}) since the geometry of the constraint may affect the flow.


In this paper, we introduce the following quantitative estimate that, slightly abusing the terminology, we call \emph{constrained \L ojasiewicz inequality}: for every stationary point $\varphi\in \mathcal K$ of $\mathcal F$ there are constants $\gamma\in\,]0,\sfrac12]$, $C>0$ such that
\begin{equation}\label{e:intro:loja2}
\big(\mathcal F(u)-\mathcal F(\varphi)\big)_+^{1-\gamma}\le C\|\nabla \mathcal F(u)\|_{\mathcal K}\quad\text{for every}\quad u \quad\text{in a neighborhood of}\quad\varphi,
\end{equation}
where $\|\nabla \mathcal F(u)\|_{\mathcal K}$ is defined as 
$$
\|\nabla \mathcal F(u)\|_{\mathcal K}:=\sup\left\{0\,,\,\sup_{v\in\mathcal K\setminus\{u\}}\frac{-(v-u)\cdot \nabla \mathcal F(u)}{\|v-u\|}\right\}.
$$
We show that \eqref{e:intro:loja2} can be used to determine the asymptotic behavior of the solutions of constrained gradient flows \eqref{e:intro:para}. Precisely, we prove that, under some mild natural assumptions on the functional $\mathcal F$ and the constraint $\mathcal K$, the constrained \L ojasiewicz inequality \eqref{e:intro:loja2} still gives \eqref{e:intro:claim} (see Proposition \ref{p:decay}). 
Thus, all the information, on the presence of the constraint and its properties, is now contained in this new  {constrained} \L ojasiewicz inequality. In particular, the estimate \eqref{e:intro:loja2} becomes 
an intrinsic property of $\mathcal K$ and $\mathcal F$, whose proof needs to be adapted to each specific situation. In particular we verify \eqref{e:intro:loja2} for the obstacle and thin obstacle problems (see Section \ref{s:loja}) thus leading to our main results on the associated parabolic flows (see Theorems \ref{t:flow1} and \ref{t:flow2}).   

Finally, let us remark that the constrained \L ojasiewicz inequality \eqref{e:intro:loja2}, combined with a new constriction based on the parabolic flow, also implies a logarithmic epiperimetric inequality at the singular points of the (time-independent) obstacle and thin-obstacle problems (see Theorem \ref{t:log}). In particular, this implies the uniqueness of the blow-up limits and the logarithmic rate of convergence of the blow-up sequences at the singular free boundary points for these problems. 
\medskip

The paper is organized as follows. In Section \ref{s:para} we prove that \eqref{e:intro:loja2} implies the claim \eqref{e:intro:claim} and in Section \ref{s:epi} we show that \eqref{e:intro:loja2} implies a logarithmic epiperimetric inequality. In Section \ref{s:loja} we prove the constrained \L ojasiewicz inequality for the obstacle and thin-obstacle problems, while in Section \ref{s:proofs} we prove our main results on the parabolic obstacle and thin-obstacle problems.

In the rest of the present section, we introduce the obstacle and the thin-obstacle problems (Subsections \ref{sub:intro:ostacolo} and \ref{sub:intro:thin}); we state our main results in Subsection \ref{sub:intro:main}, while the next subsection \label{sub:intro:loja} is dedicated to the classical \L ojasiewicz inequality for analytic functions, its applications and the relation to our results.


\subsection{\L ojasiewicz inequality on constrained domains}\label{sub:intro:loja1} 
In this subsection we use several examples of constrained and unconstrained problems in order to illustrate the main novelty of this paper: \emph{the constrained \L ojasiewicz inequality }, which is a new \L ojasiewicz-type estimate for constrained functionals. We go through the classical finite dimensional approach of \L ojasiewicz and we argue on the effects of the geometric constraint $\mathcal K$ on it. At the end of the Subsection, we make a connection with the infinite dimensional setting, by a simple model case.  
\medskip

 Let $\mathcal F:\R^N\to\R$ be a given function. For any $\xi_0\in\R^N$, consider the ODE 
\begin{equation}\label{e:ODE}
\xi'(t)=-\nabla \mathcal F(\xi(t))\quad\text{for}\quad t>0,\qquad \xi(0)=\xi_0\,.
\end{equation}
The asymptotic behavior of the global solutions starting from a point $\xi_0$ in a neighborhood of $x_0$ is a problem of major interest in several fields. One of the conditions on the function $\mathcal F$, which implies that $\xi(t)$ admits a unique limit, as $t\to\infty$, 
is the so-called \emph{\L ojasiewicz inequality}, that is:
\begin{center}{\it for every $\bar\xi\in\R^N$ critical point of $\mathcal F$, there exist an open neighborhood $U(\bar\xi)$ of $\bar\xi$,\\ and constants $C>0$ and $\gamma\in]0,\sfrac12]$, such that:}
\end{center}
\begin{equation}\label{e:loja0}
|\mathcal F(\xi)-\mathcal F(\bar\xi)|^{1-\gamma}\le C|\nabla\mathcal F(\xi)|\qquad\text{\it for every}\qquad \xi\in U(\bar\xi)\,.
\end{equation}
\noindent The following result is essentially due to \L ojasiewicz (see \cite{loja}). We sketch the proof below and for more details, we refer to the proof of Proposition \ref{p:decay}.

\begin{prop}[\L ojasiewicz decay-rate condition]\label{p:sill}
If \eqref{e:loja0} holds and $\xi:\R^+\to\R^N$ is a bounded solution of \eqref{e:ODE}, then the limit $\ds\xi_\infty:=\lim_{t\to\infty}\xi(t)$ exists and 
	\begin{equation}\label{e:decay_rate}
	|\xi(t)-\xi_\infty|\le \begin{cases} Ct^{-\frac{\gamma}{1-2\gamma}}\quad\text{if}\quad \gamma<\sfrac12\,;\\
	Ce^{-t}\quad\text{if}\quad \gamma=\sfrac12\,.
	\end{cases}
	\end{equation} 
\end{prop}

\begin{proof}
	Let $K$ be a compact set such that $\xi(t) \in K$, for every $t>0$. Furthermore, let us suppose that $K$ is contained in the neighborhood of a critical point $\bar\xi$, where the \L ojasiewicz inequality \eqref{e:loja0} does hold; this assumption is not necessary (see Proposition \ref{p:decay}), but it simplifies the proof and allows us to concentrate on the main idea. First, using \eqref{e:ODE}, we calculate 
	$$\mathcal F(\xi(t))-\min_K\mathcal F\ge \mathcal F(\xi(t))-\mathcal F(\xi(T))=-\int_t^T\xi'(s)\cdot\nabla\mathcal F(\xi(s))\,ds=\int_t^T|\nabla\mathcal F(\xi(s))|^2\,ds,$$	
	 for every $0\le t<T<\infty$. In particular, $t\mapsto \mathcal F(\xi(t))$ is non-increasing and the limit $\ds\lim_{t\to \infty} \mathcal F (\xi(t))$ exists and is finite. Let $y$ be any limit point of $\xi(t)$, as $t\to \infty$. Then, we have 
	 $$\lim_{t\to \infty}\mathcal F(\xi(t)) = \mathcal F (y)\qquad\text{and}\qquad\mathcal F(\xi(t)) > \mathcal F (y)\quad\text{for every}\quad t>0.$$
On the other hand, $t\mapsto |\nabla\mathcal F(\xi(t))|^2$ is integrable at infinity and so, there is a sequence $t_n\to\infty$ such that $|\nabla\mathcal F(\xi(t_n))|\to 0$. This, together with \eqref{e:loja0}, implies that $\ds\lim_{n\to\infty}\mathcal F(\xi(t_n))=\mathcal F(\bar\xi)=\mathcal F(y)$. 
	\noindent We now set  
	$$\ds f(t):=\mathcal F(\xi(t))-\mathcal F(\bar\xi)=\int_t^\infty|\nabla\mathcal F(\xi(s))|^2\,ds<\infty\,,\quad\text{thus}\quad f'(t)=-|\nabla\mathcal F(\xi(t))|^2.$$ 
	By \eqref{e:loja0}, we obtain the differential inequality
	$$-f'(t)\ge C\big(\mathcal F(\xi(t))-\mathcal F(\bar\xi)\big)^{2(1-\gamma)}\ge C f(t)^{2(1-\gamma)},$$
	which provides a decay rate for $f$ at infinity. Precisely, $\ds f(t)\le Ct^{-\frac{1}{1-2\gamma}}$. On the other hand, using again the equation and the Cauchy-Schwarz inequality, we get
	\begin{equation}\label{e:est:tT}
	|\xi(t)-\xi(T)|\le \int_t^T|\xi'(s)|\,ds= \int_t^T|\nabla \mathcal F(\xi(s))|\,ds\le f(t)^{\sfrac12}(T-t)^{\sfrac12}.
	\end{equation}
	Applying this inequality first to $t=2^n$ and $T=2^{n+1}$, we get that the limit $\ds\xi_\infty:=\lim_{n\to\infty}\xi(2^n)$ exists and that the rate of convergence is given precisely by \eqref{e:decay_rate}. The inequality for any $t>0$ follows again by \eqref{e:est:tT} and implies that $\xi_\infty=y$. 
\end{proof}

\begin{oss}
	It was proved by \L ojasiewicz (see \cite{loja}) that the inequality \eqref{e:loja0} holds whenever the function $\mathcal F:\R^N\to\R$ is analytic, while it is well known that the previous proposition is false in general if $\mathcal F$ is only $C^{\infty}$.
\end{oss}

Suppose next that $\mathcal F:\R^N\to\R$ is analytic, so that the \L ojasiewicz inequality holds for $\mathcal F$ at every critical point, and that $\mathcal K\subset\R^N$ is a (smooth) open convex set. Then the gradient flow $\xi:\R^+\to\mathcal \R^N$ of $\mathcal F$ in $\mathcal K$ exists and satisfies 
\begin{equation}\label{e:flow_in_K}
\xi'(t):=\begin{cases}
-\nabla \mathcal F(\xi(t))\ \text{ if }\ \xi(t)\in\mathcal K,\\
-P_{\xi(t)}(\nabla \mathcal F(\xi(t)))\ \text{ if }\ \xi(t)\in\partial \mathcal K,
\end{cases}
\end{equation}
where $P_{\xi(t)}$ is the projection on the tangent space to $\partial K$ at the point $\xi(t)$. This can be equivalently formulated as a variational inequality, that is
$$
\langle \xi'(t)+\nabla \cF(\xi(t)),\zeta-\xi(t) \rangle \geq 0\quad \mbox{for every}\quad \zeta \in \mathcal K\quad \mbox{and}\quad t>0\,.
$$
In this setting, the inequality \eqref{e:loja0} is not sufficient to conclude the convergence of the flow with a rate, since it corresponds only to the first regime of \eqref{e:flow_in_K}, when $\xi(t)\in\mathcal K$. Thus, we replace \eqref{e:loja0} with the following {\it contrained Lojasiewicz inequality}, which takes into account both regimes: 
\begin{equation}\label{e:loja1}
\big|\mathcal F(\xi)-\mathcal F(\bar \xi)\big|^{1-\gamma}\le C_L\|\nabla\mathcal F(\xi)\|_{\mathcal K}\qquad \mbox{for every }\xi\in U\cap \mathcal K\,,
\end{equation}
where we define
$$
\|\nabla \mathcal F(\xi)\|_{\mathcal K}:=\sup\left\{0\,,\,\sup_{\zeta\in\mathcal K\setminus\{\xi\}}\frac{-(\zeta-\xi)\cdot \nabla \mathcal F(\xi)}{\|\zeta-\xi\|}\right\}.
$$
Using the inequality \eqref{e:loja1} as in Proposition \ref{p:sill}, one can prove the convergence of the flow \eqref{e:flow_in_K}, under the proper assumptions (see Proposition \ref{p:decay}). However, this inequality is more difficult to prove as $\|\cdot\|_{\mathcal K}$ is smaller than the usual Euclidean norm. In particular, it vanishes on the boundary of $\mathcal K$, whenever the gradient points outside the constraint. This means that the constraint itself generates new critical points and makes the proof more challenging

 Let $\bar\xi\in \partial \mathcal K$ be a critical point for $\mathcal F$ and let $\mathcal K$ be (locally) the graph of a function $\eta:\R^{N-1}\to\R$. We consider two examples: in the first one, the decay estimate \eqref{e:loja1} does hold, while in the second one, it fails. 
		\begin{exam}[Analytic constraint]\label{exam:analytic}
		If $\eta:\R^{N-1}\to\R$ is an analytic function in a neighborhood of $\bar\xi$, then both $\mathcal F:\R^N\to \R$ and $x\mapsto \mathcal F(x,\eta(x))$ are analytic. Thus, the \L ojasiewicz inequality holds for both of them (with possibly different exponents). Taking $\gamma\in(0,\sfrac12]$ to be the smallest of the two exponents, we get that \eqref{e:loja1} holds with such $\gamma$ and so \eqref{e:decay_rate} also holds. 
		\end{exam}
		\begin{exam}[A non-analytic constraint]\label{exam:non_analytic} 
		We now consider the two dimensional problem with 
		$$\mathcal F(x,y)=y^2\ ,\qquad \eta(x)=e^{-\sfrac1{x^2}}\qquad\text{and}\qquad \xi(0)=(x_0,\eta(x_0))\in\partial \mathcal K.$$
Then, the solution $\xi$ in $\mathcal K$ exists for every $t\ge 0$ and is of the form $\xi(t)=\big(x(t),\eta(x(t))\big)$. Moreover, $\xi(t)$ converges to zero, but the decay rate is only logarithmic, so \eqref{e:decay_rate} fails. 
		\end{exam}


 The above examples show that if the constrain is analytic, then as expected \eqref{e:loja1} holds; on the other hand \emph{the non-analyticity of the constraint might make the inequality fail even if the functional itself is analytic!}
\medskip

Let us now briefly consider the infinite dimensional case. If a \L ojasiewicz-type inequality holds, then the same argument of Proposition \ref{p:sill} still works in this setting. As an example, let $u=u(t,x)\in C^1(\R^+; L^2(B_1))\cap C(\R^+;H^2(B_1)\cap H^1_g(B_1))$\footnote{For  a given function $g\in H^1(B_1)$, we use the notation $H^1_g(B_1):=\{u\in H^1(B_1)\ :\ u-g\in H^1_0(B_1)\}$.} be the solution of the heat equation 
$$\partial_t u=\Delta u\quad\text{in}\quad \R^+\times B_1,\qquad u(0,\cdot)=u_0\in L^2(B_1),\qquad u(t,x)=g(x)\quad\text{for}\quad (t,x)\in \R^+\times\partial B_1,$$
then the above decay-rate condition still holds (in $L^2(B_1)$ instead of $\R^N$) with 
$$\mathcal F(u)=\frac12\int_{B_1}|\nabla u|^2\,dx,\qquad \nabla\mathcal F(u)=\Delta u\qquad\text{and}\qquad |\nabla \mathcal F(u)|=\|\Delta u\|_{L^2(B_1)}.$$
In this case, \eqref{e:loja0} holds with $\gamma=\sfrac12$ and reads as
\begin{equation}\label{e:claim}
\int_{B_1}|\nabla u|^2\,dx-\int_{B_1}|\nabla h|^2\,dx\leq C\int_{B_1}|\Delta u|^2\,dx\,,\quad\text{for every} \quad u\in H^2(B_1)\cap H^1_g(B_1)\,,
\end{equation}
where $h$ is the harmonic function in $B_1$ with boundary datum $g$ on $\partial B_1$. Indeed,
\begin{align*}
\|\Delta u\|_{L^2}
&\geq- \frac1{\|u-h\|_{L^2}}\int_{B_1} (u-h)\Delta u\,dx=- \frac1{\|u-\phi\|_{L^2}}\int_{B_1} (u-h)\Delta (u-h)\,dx\\
&=\frac{1}{\|u-h\|_{L^2}}\int_{B_1}  |\nabla (u- h)|^2\,dx\geq C \,\left(\int_{B_1}  |\nabla( u- h)|^2\,dx\right)^{\sfrac12}\,,
\end{align*}
where in the last line we used the Poincar\'e inequality for $u-h\in H^1_0(B_1)$. Next, since $h$ is harmonic in $B_1$ and $u=h$ on $\partial B_1$, we have 
$$
\int_{B_1} |\nabla(u-h)|^2\,dx=\int_{B_1}|\nabla u|^2\,dx-\int_{B_1}|\nabla h|^2\,dx\,,
$$
which gives precisely \eqref{e:claim}.

In the seminal paper \cite{Simon0}, L. Simon proves that the estimate \eqref{e:loja0} holds for harmonic maps with values in analytic manifolds $(M,g)$, by reducing the infinite dimensional inequality to the finite dimensional one through the so-called Lyapunov-Schmidt reduction. This can be read as a minimization problem for the Dirichlet energy (which we have just seen to satisfy \eqref{e:loja0}) in the constrained analytic domain $\mathcal K=M$, that is the infinite dimensional analogues of Example \ref{exam:analytic} above. However, the situation we deal with in this paper is more similar to Example \ref{exam:non_analytic}. For instance, in the case of the parabolic obstacle problem, the gradient flow is governed by the functional 
$$\mathcal F(u)=\int_{B_1}\left(\frac12|\nabla u|^2+u\right)dx\qquad \text{and}\qquad \nabla \mathcal F(u)=-\Delta u+1,$$ 
while the constraint $\mathcal K$ is given by the (convex) set of non-negative functions in $L^2(B_1)$.
A similar reasoning as for the Dirichlet energy shows that the functional $\mathcal F$ satisfies the \L ojasiewicz inequality \eqref{e:loja0}, however \emph{\eqref{e:loja1} cannot be deduced from \eqref{e:loja0} since the constraint has non-analytic boundary!} Simon's technique, therefore, does not apply, and we have to \emph{heavily use the structure of the constrain} to conclude \eqref{e:loja1}.

\subsection{Obstacle and parabolic obstacle problems}\label{sub:intro:ostacolo} 

Let $B_1$ be the unit ball in $\R^d$ and let $g\in H^1(B_1)$ be a given non-negative function. 
We consider the functional 
$$
\cF_{\text{\tiny\sc ob}}(u):=\frac12\int_{B_1}|\nabla u|^2\,dx+\int_{B_1}u\,dx\,
$$
and the set of admissible functions 
$$\mathcal K_{\text{\tiny\sc ob}}^g:=\big\{u\in H^1(B_1)\ :\ u-g\in H^1_0(B_1),\ u\ge 0\ \text{in}\ B_1\big\}.$$
The classical obstacle problem can be written as
\begin{equation}\label{e:intro:ostacolo}
\min_{u\in\mathcal K_{\text{\tiny\sc ob}}^g}\mathcal \cF_{\text{\tiny\sc ob}}(u)
\end{equation}
and admits a unique minimizer $\phi\in \mathcal K_{\text{\tiny\sc ob}}^g$, which is also a solution of the variational inequality 
\begin{equation}\label{e:intro:ostacolo:var}
(v-\phi)\cdot \nabla \cF_{\text{\tiny\sc ob}} (\phi) \ge0\qquad\text{for every}\qquad v\in \mathcal K_{\text{\tiny\sc ob}}^g.
\end{equation}
The parabolic obstacle problem is the time-dependent counterpart of \eqref{e:intro:ostacolo:var}. We say that the function $u\in H^1\big(]0,+\infty[\,;L^2(B_1)\big) \cap L^2\big(]0,+\infty[\,; H^2(B_1)\cap\mathcal K_{\text{\tiny\sc ob}}^g\big)$ is a (global in time) solution of the parabolic obstacle problem if
\begin{equation}\label{e:intro:ostacolo:para}
\begin{cases}
\big(u'(t)+\nabla \cF_{\text{\tiny\sc ob}} (u(t))\big)\cdot (v-u(t))\ge0\quad\text{for every}\quad v\in \mathcal K_{\text{\tiny\sc ob}}^g\,,\ t>0\,,\\
u(0)=u_0\in \mathcal K_{\text{\tiny\sc ob}}^g,
\end{cases} 
\end{equation}
where $u_0$ is a given initial datum. The existence of a (strong) solution was proved in \cite{brezis}, while for the regularity we refer to the recent paper \cite{CaPeSh}.
{In Theorem \ref{t:flow1} we prove that $u(t)$ converges in $L^2(B_1)$ to the stationary solution $\phi$ with an exponential rate, while Theorem \ref{t:log} is a result on the fine structure of the (singular part of the) free boundary $\de \{\phi>0\}$. }
\medskip

Let now $(\mathcal M,g)$ be a compact connected oriented Riemannian manifold of dimension $d\ge 1$ and let $\lambda>0$ be an eigenfunction of the Laplace-Beltrami operator on $\mathcal M$. Consider the functional 
\begin{equation}
\cF_{\text{\tiny\sc ob}}^\lambda(u) = \frac12\int_{{\mathcal M}}\left(|\nabla u|^2-\lambda u^2\right) dV_g+\int_{{\mathcal M}}u\, dV_g
\end{equation} 
and the admissible set $\mathcal K_{\text{\tiny\sc ob}}^{\text{\tiny\sc m}}:=\{u\in L^2({\mathcal M})\,:\, u\geq 0\}$. 

We say that $u\in H^1\big(]0,+\infty[\,;L^2(\mathcal M)\big) \cap L^2\big(]0,+\infty[\,; H^2(\mathcal M)\cap\mathcal K_{\text{\tiny\sc ob}}^{\text{\tiny\sc m}}\big)$ is a (global in time) solution of the parabolic obstacle problem on $\mathcal M$ if
\begin{equation}\label{e:intro:ostacolo:M}
 	\begin{cases}
 	\big(u'(t)+\nabla \cF_{\text{\tiny\sc ob}}^\lambda (u(t))\big)\cdot (v-u(t))\ge0\quad\text{for every}\quad v\in \mathcal K_{\text{\tiny\sc ob}}^{\text{\tiny\sc m}}\,,\ t>0\,,\\
 	u(0)=u_0\in\mathcal K_{\text{\tiny\sc ob}}^{\text{\tiny\sc m}}.
 	\end{cases} 
 	\end{equation}
	In Theorem \ref{t:flow2} and Remark \ref{rem:M}, we will show that if the energy $\cF_{\text{\tiny\sc ob}}^\lambda (u(t))$ remains above certain critical threshold, then the solution $u(t)$ converges to a critical point of the functional $\cF_{\text{\tiny\sc ob}}^\lambda$ restricted to the convex set $\mathcal K_{\text{\tiny\sc ob}}^{\text{\tiny\sc m}}$. We notice that, contrary to \eqref{e:intro:ostacolo:para}, there might be numerous critical points of $\cF_{\text{\tiny\sc ob}}^\lambda$ in $\mathcal K_{\text{\tiny\sc ob}}^{\text{\tiny\sc m}}$ and thus, numerous candidates for the limit of $u(t)$. A priori, in such a situation the asymptotic behavior of the solution might be more complex and a limit at infinity might fail to exist. In Theorem \ref{t:flow2}, using a constrained \L ojasiewicz inequality argument, we show that the solution of \eqref{e:intro:ostacolo:M} admits a unique limit at infinity and that the presence of a whole manifold of stationary points only affects the decay rate, which is only of power type.  
\begin{oss}[On the critical points of $\cF_{\text{\tiny\sc ob}}^\lambda$ in $\mathcal K_{\text{\tiny\sc ob}}^{\text{\tiny\sc m}}$]\label{oss:ostacolo:critical}
We notice that there is more than one critical point of the functional $\cF_{\text{\tiny\sc ob}}^\lambda$ in $\mathcal K_{\text{\tiny\sc ob}}^{\text{\tiny\sc m}}$. For example, the stationary points of the unconstrained functional $\cF_{\text{\tiny\sc ob}}^\lambda$ are precisely the functions of the form 
\begin{equation}\label{e:stat:M}
u=\frac{1}{\lambda}+\phi_\lambda,
\end{equation}
where $\phi_\lambda$ is a $\lambda$-eigenvalue of the Laplace-Beltrami operator $\Delta_{\mathcal M}$, that is, 
$$-\Delta_{\mathcal M}\phi=\lambda\phi\quad\text{on}\quad\mathcal M.$$
Thus, all the positive functions $u$ of the form \eqref{e:stat:M} are stationary solutions of \eqref{e:intro:ostacolo:M}, but for a generic $\mathcal M$ and $\lambda$, there may also exist other stationary points. 
On the other hand, if $\mathcal M$ is the $(d-1)$-dimensional sphere and $\lambda=2d$, then all the stationary solutions of \eqref{e:intro:ostacolo:M} are of the form \eqref{e:stat:M} or there is a vector $\nu\in\R^d$ such that $u(x)=(x\cdot\nu)_+^2$. This is due to the fact that the $2$-homogenous extension (in the unit ball $B_1\subset\R^d$) of a stationary solution is a solution of the obstacle problem \eqref{e:intro:ostacolo} in $B_1$ and the $2$-homogeneous solutions of the obstacle problem are classified (see \cite{caffarelli_revisited}). 
\end{oss}

\subsection{Thin-obstacle and parabolic thin-obstacle problems}\label{sub:intro:thin} 
Let $d\ge 2$ and $B_1$ be the unit ball in $\R^d$. For $x\in\R^d$, we will write $x=(x',x_d)$, where $x'\in\R^{d-1}$ and $x_d\in\R$. Let $g\in H^1(B_1)$ be a given  function, which is:

$\bullet$ non-negative on $B_1\cap \{x_d=0\}$;

$\bullet$ even with respect to the hyperplane $\{x_d=0\}$, where {\it we say that a function $f:B_1\to\R$ is even with respect to the hyperplane $\{x_d=0\}$ if $f(x',x_d)=f(x',-x_d)$, for every $x=(x',x_d)\in B_1$}.

\smallskip

\noindent We consider the functional 
$$
\cF_{\text{\tiny\sc th}}(u)=\frac12\int_{B_1}|\nabla u|^2\,dx,
$$
and the admissible set 
\begin{equation*}\label{e:intro:thin}
\mathcal K_{\text{\tiny\sc th}}^g=\big\{u\in H^1(B_1)\, : \,  u-g\in H^1_0(B_1),\ u\ge 0\ \text{on}\ B_1\cap \{x_d= 0\},\ u\ \text{is even w.r.t.}\ \{x_d=0\} \big\}.
\end{equation*}
There is a unique solution $\phi\in\mathcal K_{\text{\tiny\sc th}}^g$ to the thin-obstacle problem  
\begin{equation}\label{e:intro:thin}
\min_{u\in\mathcal K_{\text{\tiny\sc th}}^g}\mathcal \cF_{\text{\tiny\sc th}}(u),
\end{equation}
and it satisfies the variational inequality 
\begin{equation}\label{e:intro:thin:var}
(v-\phi)\cdot \nabla \cF_{\text{\tiny\sc th}} (\phi) \ge0\qquad\text{for every}\qquad v\in \mathcal K_{\text{\tiny\sc th}}^g.
\end{equation}
Moreover, $\phi$ satisfies the optimality conditions
$$\Delta \phi=0\quad\text{in}\quad B_1\setminus\{x_d=0\},\qquad \frac{\partial \phi}{\partial x_d}\le 0\quad\text{and}\quad \phi\frac{\partial \phi}{\partial x_d}= 0\quad\text{on}\quad B_1\cap \{x_d=0\}.$$
We will use the notation $B_1^+:=B_1\cap\{x_d>0\}$ and we identify the even functions with their restriction on $B_1^+$. We say that  $u\in H^1\big(]0,+\infty[\,;L^2(B_1^+)\big) \cap C\big(]0,+\infty[\,; H^2(B_1^+)\cap\mathcal K_{\text{\tiny\sc th}}^g\big)$ is a (global in time) solution of the parabolic thin-obstacle problem if  
\begin{equation}\label{e:intro:thin:para}
\begin{cases}
\big(u'(t)+\nabla \cF_{\text{\tiny\sc th}} (u(t))\big)\cdot (v-u(t))\ge0\quad\text{for every}\quad v\in \mathcal K_{\text{\tiny\sc th}}^g\,,\ t> 0\,,\\
u(0)=u_0\in \mathcal K_{\text{\tiny\sc th}}^g.
\end{cases} 
\end{equation}
We recall that, for every $t>0$, $u(t)$ satisfies the optimality condition
$$u(t)\ge 0\ ,\quad\frac{\partial u(t)}{\partial x_d}\le 0\quad\text{and}\quad u(t)\frac{\partial u(t)}{\partial x_d}= 0\quad\text{on}\quad B_1\cap \{x_d=0\}.$$
The existence of a solution was first addressed in \cite{LiSt} (see also \cite{brezis}), while the latest regularity results can be found in \cite{To}.
As for the obstacle problem, we will prove in Theorem \ref{t:flow1} that the solution $u(t)$ converges exponentially to the stationary limit $\phi$, while in Theorem \ref{t:log} we will prove a logarithmic epiperimetric inequality at the singular points of the stationary free boundary $\partial \{\phi>0\}\subset\{x_d=0\}$.
\medskip

Let now $k=2m$, for some $m\in\N$, and let $\lambda=\lambda(k):=k(k+d-2)$. We consider the functional 
$$\mathcal \cF_{\text{\tiny\sc th}}^\lambda(u)=\frac12\int_{\partial B_1}\big(|\nabla u|^2-\lambda u^2\big)\,d\HH^{d-1},$$
and the admissible set 
$$\mathcal K_{\text{\tiny\sc th}}^{\mS}:=\big\{u\in L^2(\partial B_1)\,:\, u\ge 0\ \text{on}\ \partial B_1\cap \{x_d= 0\},\ u\ \text{is even w.r.t.}\ \{x_d=0\}\big\}.$$
Let $\mS=\partial B_1$ and $\mS^+=\partial B_1\cap\{x_d>0\}$. We say that  $u$ is a (global in time) solution of the parabolic thin-obstacle problem on the sphere, if $u\in H^1\big(]0,+\infty[\,;L^2(\mS^+)\big) \cap L^2\big(]0,+\infty[\,; H^2(\mS^+)\cap\mathcal K_{\text{\tiny\sc th}}^\mS\big)$ and 
\begin{equation}\label{e:intro:thin:S}
	\begin{cases}
	\big(u'(t)+\nabla \cF_{\text{\tiny\sc th}}^\lambda (u(t))\big)\cdot (v-u(t))\ge0\quad\text{for every}\quad v\in \mathcal K_{\text{\tiny\sc th}}^\mS\,,\ t>0\,\\
	u(0)=u_0\in\mathcal K_{\text{\tiny\sc th}}^\mS. 
	\end{cases} 
	\end{equation}
We will study the asymptotic behavior of the solutions to \eqref{e:intro:thin:S} in Theorem \ref{t:flow2}. 
\begin{oss}[On the critical points of $\cF_{\text{\tiny\sc th}}^\lambda$ in $\mathcal K_{\text{\tiny\sc th}}^{\mS}$]\label{oss:thin:critical} Let $m\in\N$ be fixed and let $\lambda=\lambda(2m)$.
Then, the function $\phi:\partial B_1\to\R$ is a critical point of $\cF_{\text{\tiny\sc th}}^\lambda$ in $\mathcal K_{\text{\tiny\sc th}}^{\mS}$ (in sense of \eqref{e:main:critical}) or, equivalently, a stationary solution of \eqref{e:intro:thin:S}, if and only if, the $2m$-homoegenous extension $\psi:B_1\to\R$ (in polar coordinates, $\psi(r,\theta)=r^{2m}\phi(\theta)$) is a solution of the thin-obstacle problem \eqref{e:intro:thin} in $B_1$ with trace $g=\phi$ on $\partial B_1$. On the other hand, the $2m$-homogeneous solutions of \eqref{e:intro:thin} are classified (see \cite{GaPe}) and are given precisely by the $2m$-homogeneous harmonic functions in $B_1$, non-negative on $\{x_d=0\}$. Thus, the critical points of $\cF_{\text{\tiny\sc th}}^\lambda$ in $\mathcal K_{\text{\tiny\sc th}}^{\mS}$ are the $\lambda(2m)$-eigenfunctions of the spherical Laplacian, which are non-negative on the equator $\partial B_1\cap\{x_d=0\}$.    
\end{oss}

\subsection{Main results} \label{sub:intro:main}
In this subsection, we state our main results on the parabolic (and stationary) obstacle and thin-obstacle problems. 

\begin{teo}[Asymptotics for the parabolic obstacle and thin-obstacle problems]\label{t:flow1} Let $u$ be a global (in time) solution to the parabolic obstacle problem \eqref{e:intro:ostacolo:para} (resp. the parabolic thin-obstacle problem \eqref{e:intro:thin:para}) and let $\varphi$ be the unique solution of the obstacle problem \eqref{e:intro:ostacolo} (resp. thin-obstacle problem \ref{e:intro:thin}) with the same boundary datum. Then, $u(t)$ converges to $\varphi$ strongly in $H^1(B_1)$, as $t\to\infty$, and there is a constant $C>0$ such that, for every $t\ge1$, 
\begin{equation}\label{e:flow1:convergence}
\ds\|u(t)-\varphi\|_{H^1(B_1)}\le e^{-Ct}\,.
\end{equation}
\end{teo}
 
 
 On a manifold the situation is more complicated since there is no unique minimizer. As a consequence, we can only conclude that if the flow starts close to a stationary solution and its energy is always above the energy of the solution, then it has to converge (with a rate) to a stationary solution with the same energy.

In the following theorem, $\square$ stands for $\text{\sc ob}$ (respectively, $\text{\sc th}$).
\begin{teo}[Asymptotic for the parabolic obstacle and thin-obstacle problems on the sphere]\label{t:flow2} 
Let $\mS$ be the $(d-1)$-dimensional unit sphere in $\R^d$.
Let $\mathcal S_\lambda\subset \mathcal K_{\square}^{\mS}$ be the collection of critical points of the unconstrained functional $\mathcal F_{\square}^\lambda$ in the convex set $\mathcal K_{\square}^{\mS}$, where $\lambda=2d$, if $\square=\text{\sc ob}$ and $\lambda=\lambda(2m):=2m(2m+d-2)$, if $\square=\text{\sc th}$. 

Then, there are constants $\gamma\in]0,\sfrac12[$, $\delta>0$, $E>0$ and $C>0$ such that: 
if $u$ is a solution of \eqref{e:intro:ostacolo:M} (resp. of \eqref{e:intro:thin:S} when $\square=\text{\sc th}$) on the sphere $\mS$, satisfying 
$$\text{\rm dist}_{L^2}(u_0,\mathcal S_\lambda)\le\delta\,,\quad\mathcal F_{\square}^\lambda(u_0)-\mathcal F_{\square}^\lambda(\mathcal S_\lambda)\le E\quad\text{and}\quad \mathcal F_{\square}^\lambda(u(t))>\mathcal F_{\square}^\lambda(\mathcal S_\lambda),\quad\text{for every}\quad t>0,$$
then there is a critical point $\varphi\in \mathcal S_\lambda\subset \mathcal K_{\square}^{\mS}$ of $\mathcal F_{\square}^\lambda$ such that, for every $t\ge 1$,
\begin{equation}\label{e:in_culo_decay2}
\|u(t)-\varphi\|_{H^1(\mS)}\le Ct^{-\frac{\gamma}{1-2\gamma}}.
\end{equation}
\end{teo}
\begin{oss}\label{rem:M}
In the case of the obstacle problem, the sphere $\mS$ can be replaced by a compact Riemmanian manifold $\mathcal M$ and $\lambda$ can be taken to be any eigenvalue of the Laplace-Beltrami operator on $\mathcal M$. In this setting, we take the set $\mathcal S$ to be the set of critical points of the unconstrained functional $\mathcal F_{\text{\tiny\sc ob}}^\lambda$, which are positive (so, they lie inside the convex set $\mathcal K_{\text{\tiny\sc ob}}^{\text{\tiny\sc m}}$). In this case, the conclusion is that $u(t)$ converges to a function $u_\infty\in \mathcal K_{\text{\tiny\sc ob}}^{\text{\tiny\sc m}}$, which is a critical point for $\mathcal F_{\text{\tiny\sc ob}}^\lambda$ in $\mathcal K_{\text{\tiny\sc ob}}^{\text{\tiny\sc m}}$ in the sense of \eqref{e:main:critical}. Moreover, we have the estimates 
\begin{equation}\label{e:in_culo_decay3}
\|u(t)-\varphi\|_{L^2(\mathcal M)}\le Ct^{-\frac{\gamma}{1-2\gamma}}\qquad\text{and}\qquad \mathcal F_{\text{\tiny\sc ob}}^\lambda(u(t))-\mathcal F_{\text{\tiny\sc ob}}^\lambda(\mathcal S)\le Ct^{-\frac{1}{1-2\gamma}}\,. 
\end{equation}
This is a consequence of Proposition \ref{p:decay} and Proposition \ref{prop:thin:loja}.
\end{oss}

Our interest in the parabolic problems \eqref{e:intro:ostacolo:M} and \eqref{e:intro:thin:S} on the sphere is two-folded: on the one hand they are the natural generalizations of the respective parabolic problems in $B_1\subset \R^d$ studied in Theorem \ref{t:flow1} above; on the other hand, they are strictly related to the study of uniqueness of blow-ups at singular points for the time-independent problem, where the radial direction is treated as time. This observation goes back to Simon in \cite{Simon0} for stationary varifolds and harmonic maps, and we make it explicit in the context of minimizers of the obstacle and thin-obstacle problems by proving the following logarithmic epiperimetric inequality. Before stating it, we introduce the notation
$$
\G_{\text{\tiny\sc ob}}(u):=\cF_{\text{\tiny\sc ob}}(u)-\int_{\de B_1}u^2\,d\HH^{d-1}
\qquad \mbox{and}\qquad
\G_{\text{\tiny\sc th}}(u):=\cF_{\text{\tiny\sc th}}(u)-m\int_{\de B_1}u^2\,d\HH^{d-1}\,,
$$
where $m\in \N$. For the obstacle problem, we define the set of stationary points $\mathcal S_{\text{\tiny\sc ob}}$ on the sphere $\partial B_1\subset\R^d$ as the traces of all global two-homogeneous non-flat solutions of the obstacle problem in $\R^d$ (see \cite{caffarelli_revisited}), that is, 
$$\mathcal S_{\text{\tiny\sc ob}}:=\big\{Q_A \colon \partial B_1 \to \R \,:\, Q_A(x) = x \cdot Ax,\,\text{ $A$  symmetric non-negative matrix with }{\rm{tr}} A = \sfrac12\big\}.$$
We notice that $\G_{\text{\tiny\sc ob}}$ is constant on $\mathcal S_{\text{\tiny\sc ob}}$ and we set 
$\Theta:=\G_{\text{\tiny\sc ob}}(\mathcal S_{\text{\tiny\sc ob}})$.

\begin{teo}[Log-epiperimetric inequality for obstacle and thin-obstacle at singular points]\label{t:log} Let $d\ge 2$. The following logarithmic epiperimetric inequalities hold.
	\begin{itemize}
 		\item[(OB)] There are dimensional constants $\delta>0$ and $\eps>0$ such that for every non-negative function $c\in H^1(\partial B_1)$, with $2$-homogeneous extension $z$ on $B_1$, satisfying   
 		$$\text{dist}_{L^2(\partial B_1)}\left(c,\mathcal S_{\text{\tiny\sc ob}}\right)\le \delta\qquad\text{and}\qquad  \G_{\text{\tiny\sc ob}}(z)-\Theta\le 1,$$
 		{\flushright there is a non-negative function $h\in \mathcal K_{\text{\tiny\sc ob}}^{c}$ satisfying the inequality }
 		\begin{equation}\label{e:epi_flat_point_sing}
 		\G_{\text{\tiny\sc ob}}(h)-\Theta \le \big(\G_{\text{\tiny\sc ob}}(z)-\Theta\big)\big(1-\eps\big|\G_{\text{\tiny\sc ob}}(z)-\Theta\big|^{\gamma} \big)\,,\qquad \mbox{where}\quad
 		\begin{cases}
 		\gamma=0 & \mbox{if }d=2\\
 		\gamma=\frac{d-1}{d+3}  & \mbox{if }d\geq3
 		\end{cases}.
 		\end{equation}
		\item[(TH)] Let $d\ge 2$ and $m\in \N$. For every function $c\in H^1(\partial B_1)\cap \mathcal K_{\text{\tiny\sc th}}^{\mS}$ such that 
		\begin{equation}\label{e:orchecattive}
		\int_{\partial B_1}c^2\,d\HH^{d-1}\le 1\qquad\text{and}\qquad |\mathcal G_{\text{\tiny\sc th}}(z)|\le 1\,,
		\end{equation}
		there are a constant $\eps=\eps(d,m)>0$ and a function $h\in \mathcal K_{\text{\tiny\sc th}}^c$ satisfying 
		\begin{equation}\label{e:epi:2m}
		\G_{\text{\tiny\sc th}}(h)\le \G_{\text{\tiny\sc th}}(z) \big(1-\eps \,|\G_{\text{\tiny\sc th}}(z)|^{\gamma}\big),\qquad \mbox{where}\quad \gamma:=\frac{d-2}{d}\,.
		\end{equation}
	\end{itemize}
\end{teo}

The epiperimetric and the logarithmic epiperimetric inequalities are part of the same family of quantitative estimates on the energy of the homogeneous functions. They are used to obtain regularity of the free boundaries with modulus of continuity, which is H\"older in the first case and logarithmic in the latter. This homogeneity-improvement argument was pioneered by Reifenberg in \cite{Reif2} in the context of minimal surfaces and several authors used it in the context of minimal surfaces and free boundary problems (see
\cite{Wh,taylor1,taylor,CoMi,DSS1,Weiss,SpVe,CoSpVe1,CoSpVe2, esv, esv2}). 

Even if the epiperimetric inequalities, and the methods to deduce regularity from them, might seem quite similar, the methods to prove them are very different. The first epiperimetric inequality for a free boundary (obstacle) problem was proved by Weiss in \cite{Weiss}, where he used an argument by contradiction, which was then applied to different obstacle problems in \cite{FoSp,GaPeVe}; this is a powerful method, which allows to prove the regularity of the flat free boundaries, but it cannot be applied to the singular points, where the integrability fails. In \cite{SpVe}, we used a new, direct approach, inspired by the idea of Reifenberg, which consists in the explicit construction of the competitor starting from the Fourier expansion of the trace; we later used this idea in \cite{CoSpVe1} and \cite{CoSpVe2} to prove the logarithmic epiperimetric inequalities from Theorem \ref{t:log} (OB) and (TH).  

In this paper we give a new, different type of proof, which is constructive (that is, we construct a competitor), but not direct (the competitor is not explicit). We use stopped parabolic flows for a functional on the unit sphere $\partial B_1$ and, identifying the time with the radial direction, we reparametrize it over the spheres $\partial B_r$ to obtain a function defined on the unit ball, which lives in the constraint domain and has smaller energy. This is a general abstract procedure, inspired from \cite{esv} and described in detail in Section \ref{s:epi}.
 
Finally, we recall that  the structure of the singular part of the free boundary of minimizers of the obstacle and thin-obstacle problems was studied by several authors; we refer to \cite{caffarelli_revisited}, \cite{CaRi}, \cite{Monneau}, \cite{FiSe} and \cite{CoSpVe1}, for the obstacle problem, and \cite{GaPe}, \cite{CoSpVe2} for the thin-obstacle problem. We also note that, the uniqueness of the blow-up and the logarithmic modulus of continuity follow directly by the lograithmic epiperimetric inequality (Theorem \ref{t:log}), exactly as in \cite{CoSpVe1,CoSpVe2}.

\section{Asymptotic behavior for parabolic variational inequalities} \label{s:para}

In this section we prove that solutions of parabolic variational inequalities of energies satisfying a \emph{constrained \L ojasiewicz inequality} converge at infinity to a stationary solution of the energy. In order to state the main result we need to introduce some notations.

Let $\mathcal H$ be a Hilbert space with scalar product 
$$u\cdot v= \langle u,v\rangle =\langle u,v\rangle_{\mathcal H}\qquad\text{for every}\qquad u,v\in\mathcal H,$$
and induced norm $\|u\|=\|u\|_\mathcal H=\sqrt{u\cdot u}$. Let $\mathcal W\subset \mathcal H$ be a linear subspace, $\mathcal F:\mathcal W\to\R$ and $\nabla\mathcal F:\mathcal W\to \mathcal H$ be continuous (possibly non-linear) functionals such that 
\begin{equation}\label{e:main:derivataF}
\mathcal F(u+tv)=\mathcal F(u)+t\, v\cdot \nabla \mathcal F(u)+o(t),\quad\text{for every}\quad u,v\in \mathcal W.
\end{equation}
Let $\mathcal K\subset \mathcal H$ be a convex subset of $\mathcal H$. We will suppose that $\mathcal K\cap \mathcal W$ is dense in $\mathcal K$. 

\noindent For every $u\in\mathcal K\cap \mathcal W$, we define
\begin{equation}\label{e:main:Knorm}
\|\nabla \mathcal F(u)\|_{\mathcal K}:=\sup\left\{0\,,\,\sup_{v\in\mathcal K\setminus\{u\}}\frac{-(v-u)\cdot \nabla \mathcal F(u)}{\|v-u\|}\right\}.
\end{equation}
We say that $u\in\mathcal K\cap \mathcal W$ is a \emph{critical point of the functional $\mathcal F$ in $\mathcal K$}, if we have
\begin{equation}\label{e:main:critical}
(v-u)\cdot \nabla \mathcal F(u)\ge 0\quad\text{for every}\quad v\in \mathcal K. 
\end{equation}
The following is a simple exercise left to the reader.
\begin{lm}[Critical points]
Let $u\in\mathcal K\cap \mathcal W$. Then, the following are equivalent: 
\begin{enumerate}[(i)]
\item $u$ is a critical point for $\mathcal F$ in $\mathcal K$; 
\item for every $v\in\mathcal K$, the function $t\mapsto \mathcal F((1-t)u+tv)$ is differentiable in zero and $$\ds\frac{d}{dt}\Big\vert_{t=0^+}\mathcal F((1-t)u+tv)\ge 0\,;$$ 
\item $\|\nabla \mathcal F(u)\|_{\mathcal K}=0$.
\end{enumerate}
\end{lm}

The reader may keep in mind the following guiding example: 

\begin{exam}
 Let $\mathcal H=L^2(B_1)$, $\mathcal K=\{u\in L^2(B_1)\ :\ u\ge 0\ \text{on}\  B_1\}$, $\mathcal W=H^2(B_1)\cap H^1_g(B_1)$, where $H^1_g(B_1):=\{u\in H^1_g(B_1)\ :\ u=g\in H^1(B_1)\ \text{on}\  \partial B_1\}$; and let 
$$
\mathcal F(u)=\frac12\int_{B_1}|\nabla u|^2\,dx+\int_{B_1}u\,dx\,.
$$ 
Thus, the critical points of $\mathcal F$ in $\mathcal K$ are precisely the solutions of the obstacle problem in $B_1$ with boundary datum $g$ on $\partial B_1$. 
\end{exam}

\begin{definition}[Parabolic variational inequalities]\label{def:main:existence}
Let $u_0\in\mathcal K$ and $T\in\,]0,\infty]$. 
We say that the function $u:[0,T[\,\to\mathcal K$ is a \emph{(strong) solution} (global, if $T=+\infty$,) of the parabolic variational inequality 
\begin{equation}\label{e:main:flow1K}
\begin{cases}
\big(u'(t)+\nabla \mathcal F(u(t))\big)\cdot (v-u(t))\ge0\quad\text{for every}\quad v\in \mathcal K\,,\ t\in[0,T[\,,\\
u(0)=u_0,
\end{cases} 
\end{equation}
if it satisfies the following conditions:
\begin{enumerate}[(i)]
\item {\it (continuity)} $u\in C([0,T[\,; \mathcal H)$ and $u(0)=u_0$; 
\item {\it (regularity in time)} $u\in H^1_{loc}\big(]0,T[\,;\mathcal H\big)$;  in particular, $u:\,]0,+\infty[\,\to \mathcal H$ is differentiable in almost every $t>0$;
\item {\it (regularity in space)} $u\in C(]0,T[\,;\mathcal W\cap\mathcal K)$;
\item {\it (variational inequality)} for every $v\in\mathcal K$ and almost-every $t>0$ we have 
\begin{equation}\label{e:main:parabolic}
	\left\langle u'(t)+\nabla \mathcal F(u(t)), v-u(t) \right\rangle_{\mathcal H}\ge 0.
	\end{equation}
\end{enumerate}
\end{definition}
\begin{oss}[Stationary solutions]
We notice that if $\varphi\in\mathcal K\cap\mathcal W$ is a critical point of $\mathcal F$ in $\mathcal K$, then $u(t)\equiv\varphi$ is a solution of \eqref{e:main:flow1K}. On the other hand, if $u(t)\equiv\varphi$ is a solution of \eqref{e:main:flow1K}, then the variational inequality \eqref{e:main:parabolic} implies that $\varphi$ is a critical point.
\end{oss}

We will need the following property of strong global solutions of parabolic variational inequalities..

\begin{lm}\label{l:puppa1}
	Let $u$ be a strong solution of \eqref{e:main:flow1K}. The following properties are true.
	\begin{itemize}
		\item[(i)] $\ds \|u'(t)\|^2= - u'(t)\cdot \nabla \mathcal F(u(t))$ for almost-every $t\in\,]0,T[\,$.
		\item[(ii)] The function $t\mapsto \mathcal F(u(t))$ is non-increasing.
		\item[(iii)] $\|u'(t)\|=\|\nabla \mathcal F(u(t))\|_\mathcal K$.
	\end{itemize}
\end{lm}

\begin{proof} To prove (i) we notice that taking  $t>0$, $h>0$ and $v:=u(t+h)$ in \eqref{e:main:parabolic}, we get 
	\begin{align*}
	0\le \frac1h\big(u(t+h)-u(t)\big)\cdot \big( u'(t)+\nabla \mathcal F(u(t))\big).
	\end{align*}
	Passing to the limit as $h\to 0$, we obtain 
	$$\ds \|u'(t)\|^2\ge -u'(t)\cdot\nabla \mathcal F(u(t)).$$
	Conversely, taking $h>0$ and $v:=u(t-h)$ in \eqref{e:main:parabolic}, we get the opposite inequality. Combining the two estimates, we deduce 
	\begin{equation}\label{e:para:ost:u'}
	\ds \|u'(t)\|^2= - u'(t)\cdot \nabla \mathcal F(u(t))\quad\text{for almost-every}\quad t>0.
	\end{equation}
	(ii) is an immediate consequence of (i).  Indeed, let us first notice that \eqref{e:main:derivataF} implies $$\frac{d}{ds}\mathcal F\big(s\,u(t+h)+(1-s)u(t)\big)=\big(u(t+h)-u(t)\big)\cdot\nabla \mathcal F\big(s u(t+h)+(1-s)u(t)\big).$$
	Thus, 
	$$\frac1h\big(\mathcal F(u(t+h))-\mathcal F(u(t))\big)=\int_0^1\frac1h\big(u(t+h)-u(t)\big)\cdot\nabla \mathcal F\big(s\, u(t+h)+(1-s)u(t)\big)\,ds.$$
	By the continuity of $\nabla \mathcal F$ (in $\mathcal W$), the continuity of $u$ (in $\mathcal W$) and the differentiability of $u$ (in $\mathcal H$), we can pass to the limit , as $h\to0$, in all the points $t$, where $u(t)$ is diffferentiable:
	\begin{equation}\label{e:main:derivataF(u(t))}
	\frac{d}{dt}\mathcal F(u(t))=u'(t)\cdot\nabla \mathcal F(u(t))=-\|u'(t)\|^2\le 0. 
	\end{equation}
	We finally come to (iii). We consider two cases: 
	
	\noindent {\it Case 1. Suppose that $\|u'(t)\|=0$.} Then \eqref{e:main:parabolic} and \eqref{e:main:Knorm} imply that $\|\nabla\mathcal F(u(t))\|_{\mathcal K}=0$.  
	
	\noindent {\it Case 2. Let $\|u'(t)\|>0$.} Then we have
	\begin{align*}
	\sup_{v\in\mathcal K}\frac{-(v-u(t))\cdot \nabla\mathcal F(u(t))}{\|v-u(t)\|}&\ge\lim_{h\to 0^+} \frac{-\frac1h\big(u(t+h)-u(t)\big)\cdot\nabla \mathcal F(u(t))}{\frac1h\|u(t+h)-u(t)\|}\\
	&=\frac{-u'(t)\cdot\nabla \mathcal F(u(t))}{\|u'(t)\|}=\|u'(t)\|.
	\end{align*}
	On the other hand, the parabolic variational inequality \eqref{e:main:parabolic} implies that
	$$\|u'(t)\|\ge \frac{u'(t)\cdot (v-u(t))}{\|v-u(t)\|}\ge \frac{-(v-u(t))\cdot\nabla \mathcal F(u(t))}{\|v-u(t)\|}\quad\text{for every}\quad v\in\mathcal K.$$ 
	Taking the supremum over $v\in\mathcal K$, we finally get 
	$$\|u'(t)\|=\sup_{v\in\mathcal K}\frac{-(v-u(t))\cdot \nabla\mathcal F(u(t))}{\|v-u(t)\|}=\|\nabla \mathcal F(u(t))\|_{\mathcal K},$$
	where the equality follows by the fact that $\|u'(t)\|\ge 0$.
	\end{proof}

\begin{definition}[Continuity with respect to the intial datum]\label{def:main:continuity}
We say that the \emph{flow \eqref{e:main:flow1K} of $-\nabla \mathcal F$ in $\mathcal K$ depends continuously on the initial datum}, if for every $t>0$ and every $\eps>0$ there is $\delta>0$ such that: if $u_0,v_0\in\mathcal K$ are such that $\|u_0-v_0\|\le \delta$, then $\|u(s)-v(s)\|\le \eps$, for every $s\in[0,t]$ (for which both flows are defined). 	
\end{definition}

We notice that the continuity of the flow of $-\nabla \mathcal F$ in $\mathcal K$ essentially boils down to the continuity of the flow of $-\nabla \mathcal F$  without any constraint. Indeed we have the following simple lemma.

\begin{lm}\label{l:cont}
If $\nabla\mathcal F:\mathcal W\to\mathcal H$ is a linear function and there is a constant $\lambda>0$ such that 
	\begin{equation}\label{e:main:growth_condition}
	- u\cdot \nabla \mathcal F(u)\le \lambda \|u\|^2\qquad\text{for every}\qquad u\in \mathcal W,
	\end{equation}
	then the flow \eqref{e:main:flow1K} of $-\nabla \mathcal F$ in $\mathcal K$ depends continuously on the initial datum.
\end{lm}

\begin{proof}
If $u$ and $v$ are two solutions of \eqref{e:main:flow1K} on $[0,T[$, then for every $t\in [0,T[$ we have 
$$
\left\langle u'(t)+\nabla \mathcal F(u(t)), v(t)-u(t)\right\rangle\ge 0\qquad\text{and}\qquad
\left\langle v'(t)+\nabla \mathcal F(v(t)), u(t)-v(t)\right\rangle\ge 0.
$$ 
Thus, by \eqref{e:main:parabolic}, we get
\begin{align*}
\frac12\frac{d}{dt}\|u(t)-v(t)\|^2&=\big\langle u'(t), u(t)-v(t)\big\rangle+\big\langle v'(t), v(t)-u(t)\big\rangle\\
&\le \big\langle \nabla \mathcal F(u(t)), v(t)-u(t)\big\rangle+\big\langle \nabla \mathcal F(v(t)), u(t)-v(t)\big\rangle.
\end{align*}
By the linearity of $\nabla\mathcal F$, we get 
\begin{equation}\label{e:main:pre_gronwall}
\frac12\frac{d}{dt}\|u(t)-v(t)\|^2\le -\big\langle u(t)-v(t),\nabla \mathcal F\big(u(t)-v(t)\big)\big\rangle.
\end{equation}
Now, \eqref{e:main:pre_gronwall} implies that 
$$\frac{d}{dt}\|u(t)-v(t)\|^2\le 2\lambda \|u(t)-v(t)\|^2,$$
so, we can take $\delta= e^{2\lambda t}\eps$ in Definition \ref{def:main:continuity}. 
\end{proof}

In the sequel we will denote by \emph{$\mathcal S\subset \mathcal K\cap\mathcal W$ a subset of the set of critical points of $\cF$ in $\mathcal K$}.

\begin{definition}[Constrained \L ojasiewicz inequality]\label{def:main:loja}
We say that the functional $\mathcal F$ has the \emph{constrained \L ojasiewicz property on $\mathcal S$} if there are constants $\gamma\in\,]0,\sfrac12]$, $C_L>0$, $\delta_L>0$ and $E_L>0$, depending on $\mathcal S$, such that for every critical point $\mathcal \varphi \in \mathcal S$ the following inequality holds 
\begin{equation}\label{e:main:loja}
\big(\mathcal F(u)-\mathcal F(\varphi)\big)_+^{1-\gamma}\le C_L\|\nabla\mathcal F(u)\|_{\mathcal K},
\end{equation}
for every $u\in \mathcal K\cap\mathcal W$ such that $\|u-\varphi\|\le \delta_L$ and $\mathcal F(u)-\mathcal F(\varphi)\le E_L$.
\end{definition}

The following lemma is an easy exercise left to the reader.

\begin{lm}[Stationary points and \L ojasiewicz inequality] If $\mathcal F$ has the constrained \L ojasiewicz property on the subset of critical points $\mathcal S$, then it is locally constant on $\mathcal S$. 
\end{lm}

Given $\varphi \in \mathcal S$, we will denote by 
$$
\mathcal S_\varphi:=\{\psi \in \mathcal S\,:\,\cF(\psi)=\cF(\varphi) \}
$$  
and write $\mathcal F(\mathcal S_\varphi):=\mathcal F(\varphi)$. If there is only one such energy level, we will drop the index  $\varphi$.
We are now able to state the main result of this section. We will use the notation 
$$\text{dist}(u,\mathcal S_\varphi)=\inf_{\psi\in \mathcal S_\varphi}\|u-\psi\|,$$ 
and by {\it neighborhood of} $\mathcal S_\varphi$ (in $\mathcal K$) we will mean a set (containing a set) of the form $$\{u\in \mathcal K\ :\ \text{dist}(u,\mathcal S_\varphi)<\delta\},$$
for some $\delta>0$.

\begin{prop}\label{p:decay}
Let $\mathcal H$, $\mathcal W$, $\mathcal K$, $\mathcal F:\mathcal W\to\R$ and $\nabla\mathcal F:\mathcal W\to\mathcal H$ be as above. Let $\mathcal S\subset \mathcal W\cap\mathcal K$ be a subset of the set of critical points of $\cF$ in $\mathcal K$. Let $\varphi \in \mathcal S$ and $\mathcal S_\varphi$ be as above.
Suppose that there is a neighborhood of $\mathcal S_\varphi\subset \mathcal W\cap\mathcal K$ , where:

(a) the flow \eqref{e:main:flow1K} of $-\nabla \mathcal F$ in $\mathcal K$ depends continuously on the initial datum (Definition \ref{def:main:continuity}); 

(b) $\cF$ has the constrained \L ojasiewicz property on $\mathcal S_\varphi$ (Definition \ref{def:main:loja}).

\noindent Then, there are constants $\delta>0$, $E>0$ and $C>0$ such that: if $u_0\in\mathcal K\cap \mathcal W$ and $u(t)$ is a global solution (with initial datum $u_0$) satisfying
$$
\text{\rm dist}(u_0,\mathcal S_\varphi)\le\delta\,,\quad\mathcal F(u_0)-\mathcal F(\mathcal S_\varphi)\le E\quad\text{and}\quad \mathcal F(u(t))>\mathcal F(\mathcal S_\varphi),\quad\text{for every}\quad t>0,
$$
then there is a function $u_\infty\in\mathcal \mathcal W\cap \mathcal K$ such that $u(t)$ converges to $u_\infty$ and, if $\gamma<\sfrac12$, then
\begin{equation}\label{e:main:decay_rate}
\|u(t)-u_\infty\|\le Ct^{-\frac{\gamma}{1-2\gamma}}\qquad\text{and}\qquad \mathcal F(u(t))-\mathcal F(\mathcal S_\varphi)\le Ct^{-\frac{1}{1-2\gamma}},
\end{equation}
for every $t\ge 1$, while if $\gamma=\sfrac12$, then the decay is exponential:
\begin{equation}\label{e:main:decay_rate2}
\|u(t)-u_\infty\|\le e^{-Ct}\qquad\text{and}\qquad \mathcal F(u(t))-\mathcal F(\mathcal S_\varphi)\le e^{-Ct}.
\end{equation}
\end{prop}
	
\begin{proof}
	Let $\delta_L>0$ be the constant from Definition \ref{def:main:loja} and set for simplicity $\mathcal S=\mathcal S_\varphi$. We will prove that there is $\delta\in(0,\delta_L)$ with the following property: for any $u_0\in \mathcal K$ such that $\text{dist}(u_0,\mathcal S)<\delta$, the solution $u(t)$ of \eqref{e:main:flow1K} exists for every $t>0$ and $\text{dist}(u(t),\mathcal S)<\delta_L$. The decay rate \eqref{e:main:decay_rate} will be then a consequence of the \L ojasiewicz inequality. We next suppose that $u$ satisfies the inequality $\mathcal F(u(t))> \mathcal F(\mathcal S)$, for every $t>0$. Let us first recall that, by Lemma \ref{l:puppa1} (ii), the map $t\mapsto \mathcal F(u(t))$ is non-increasing. Thus, the condition $\mathcal F(u(t))-\mathcal F(\mathcal S)\le E$ is automatically satisfied for every $t\ge 0$, once it holds for $t=0$.  In particular, we can simply take $E$ to be the constant $E_L$ from the \L ojasiewicz inequality (and if $E_L=+\infty$, then also $E=+\infty$).  
	
	Let $0<t<T<+\infty$ be given. Then, we have 
\begin{equation}\label{e:main:mainest0}
	\int_t^T \|u'(s)\|^2\,ds=-\int_t^T u'(s)\cdot \nabla \mathcal F(u(s))\,ds=\mathcal F(u(t))-\mathcal F(u(T))\le \mathcal F(u(t))-\mathcal F(\mathcal S).
	\end{equation}
	In particular, we obtain that the function $s\mapsto\|u'(s)\|$ is square integrable at infinity and 
	\begin{equation}\label{e:main:mainest}
	\int_t^{+\infty} \|u'(s)\|^2\,ds\le \mathcal F(u(t))-\mathcal F(\mathcal S)\qquad \text{for every}\qquad t>0.
	\end{equation}
	Let now $T\in]0,+\infty]$ be such that $\text{dist}(u(t),\mathcal S)<\delta_L$, for every $t\in]0,T[$. Then, for any $0<t<T$, we can estimate the right-hand side of \eqref{e:main:mainest} by the \L ojasiewicz inequality \eqref{e:main:loja}
\begin{equation}\label{e:main:mainest_loja}
	\int_t^{+\infty} \|u'(s)\|^2\,ds\le \mathcal F(u(t))-\mathcal F(\mathcal S)\le C\|\nabla\mathcal F (u(t))\|_{\mathcal K}^{\frac{1}{1-\gamma}}= C\|u'(t)\|^{\frac{1}{1-\gamma}},
\end{equation}
where in the last equality we used the identity (iii) of Lemma \ref{l:puppa1}. Here and in what follows $C$ will denote {\it any constant} depending only on the constant $C_L$ from the \L ojasiewicz inequality \eqref{e:main:loja} and the exponent $\gamma$. Setting $$\ds \xi(t):=\int_t^{+\infty} \|u'(s)\|^2\,ds,$$
we get that 
	\begin{align}
	-\xi'(t)=\|u'(t)\|^2\ge C\left(\int_t^{+\infty} \|u'(s)\|^2\,ds\right)^{2(1-\gamma)}=C\xi(t)^{2(1-\gamma)}.\label{e:main:xi_est}
	\end{align}
	From now on, we consider the case $\gamma<\sfrac12$. 
	Thus, we get that the function 
	$t\mapsto \left(\xi(t)^{2\gamma-1}-Ct\right)$
	is non-decreasing on $]0,T[$. Thus, for every $0<s<t<T$, we have the estimate
			\begin{equation}\label{e:main:xi_est}
		\xi(t)\le \left(\xi(s)^{-(1-2\gamma)}+C(t-s)\right)^{-\frac{1}{1-2\gamma}}.
			\end{equation}
		Now, let $0<s< t_1< t_2<T$. Then 
		\begin{align*}
		\|u(t_2)-u(t_1)\|&\le \left\|\int_{t_1}^{t_2}u'(\tau)\,d\tau\right\|\le \int_{t_1}^{t_2}\|u'(\tau)\|\,d\tau\\
		&\le \left(\int_{t_1}^{t_2}\|u'(\tau)\|^2\,d\tau\right)^{\sfrac12}(t_2-t_1)^{\sfrac12}\le \xi(t_1)^{\sfrac12}(t_2-t_1)^{\sfrac12}\\
&\le \left(\xi(s)^{-(1-2\gamma)}+C(t_1-s)\right)^{-\frac1{2(1-2\gamma)}}(t_2-t_1)^{\sfrac12}\\
		&\le C (t_1-s)^{-\frac1{2(1-2\gamma)}}(t_2-t_1)^{\sfrac12}.
		\end{align*}
	Taking $k\ge 1$ and applying the above inequality to $t_2=2^{k+1}$, $t_1=2^k$ and $s\le 2^{k-1}$, we get 
		\begin{align*}
		\big\|u(2^{k+1})-u(2^k)\big\|\le 
C\,2^{-\frac{k\gamma}{1-2\gamma}}.
		\end{align*}
		In particular, for every $m> n$ such that $2^m<T$, we obtain 
		\begin{equation}\label{e:main:cauchy_dia}
		\big\|u(2^m)- u(2^n)\big\|\le \sum_{k=n}^\infty \big\|u(2^{k+1})-u(2^k)\big\|= \frac{C}{1-2^{-\frac{\gamma}{1-2\gamma}}}\,2^{-\frac{n\gamma}{1-2\gamma}}.
		\end{equation}
		On the other hand, if $t<T$ and $2^m\le t<2^{m+1}$, then 
		\begin{equation}\label{e:main:cauchy_t}
		\big\|u(t)- u(2^m)\big\|\le C (2^m)^{-\frac1{2(1-2\gamma)}}(t-2^m)^{\sfrac12}\le \sqrt2 C\,(2^m)^{-\frac{\gamma}{1-2\gamma}}.
		\end{equation}
		Thus, for every $n\ge 1$ and $2^n\le t<T$, we obtain 
		\begin{equation}\label{e:main:cauchy}
		\big\|u(t)- u(2^n)\big\|\le C\,2^{-\frac{n\gamma}{1-2\gamma}}.
		\end{equation}
First of all, we choose $n$ such that $\ds C\, 2^{-\frac{n\gamma}{1-2\gamma}}< \frac{\delta_L}2$. Next, we choose $\delta>0$ such that 
$$\text{dist}(u(t),\mathcal S)< \frac{\delta_L}2,\quad\text{for every}\quad t\in]0,2^n[$$(such a constant exists due to the continuity with respect to the initial datum). In particular, this implies that we can take $T=+\infty$. 

We now use again \eqref{e:main:cauchy_dia}, this time for every $m>n$, obtaining that the limit 
$$u_\infty:=\lim_{n\to\infty}u(2^n),$$
exists and is such that $\ds\|u_\infty-u(2^n)\|\le C\,2^{-\frac{n\gamma}{1-2\gamma}}$. Finally, \eqref{e:main:cauchy} implies that 
$$\lim_{t\to\infty}u(t)=u_\infty\qquad\text{and}\qquad \|u_\infty-u(t)\|\le C\,t^{-\frac{\gamma}{1-2\gamma}}.$$ 

In order to obtain the energy decay (the second inequality in \eqref{e:main:decay_rate}) we notice that, by the monotonicity of $t\mapsto \mathcal F(u(t))$, the inequality \eqref{e:main:mainest_loja} and the integrability of $t\mapsto \|u'(t)\|^2$, we get that 
$\ds\lim_{t\to \infty}\mathcal F(u(t))=\mathcal F(\mathcal S).$
Thus, passing to the limit as $T\to\infty$ in \eqref{e:main:mainest0}, we get that \eqref{e:main:mainest} holds with an equality. Now, the decay rate of the energy is a consequence of \eqref{e:main:xi_est}. Finally, let $\psi(t)$ be the solution of \eqref{e:main:flow1K} with initial datum $u_\infty$. The continuity with respect to the initial datum and \eqref{e:main:decay_rate} imply that $\psi(t)$ is stationary and so, $u_\infty$ is a critical point of $\mathcal F$ in $\mathcal K$. 
\smallskip

The case $\gamma=\sfrac12$ is analogous. Moreover, since \eqref{e:main:loja} with $\sfrac12$ implies \eqref{e:main:loja} with any $\gamma<\sfrac12$ (up to changing $C_L$), we can actually use what we already know from $\gamma<\sfrac12$. In particular, \eqref{e:main:loja} can be applied all along the flow and $u(t)$ converges to $u_\infty$. Thus, \eqref{e:main:xi_est} implies that 
$$\mathcal F(u(t))-\mathcal F(\mathcal S_\varphi)=\xi(t)\le \xi(0)e^{-Ct}=\big(\mathcal F(u_0)-\mathcal F(\mathcal S_\varphi)\big)\,e^{-Ct},$$
which gives the second part of \eqref{e:main:decay_rate2}. In order to get the first part, we notice that, choosing $\gamma$ sufficiently close to $\sfrac12$, we get from \eqref{e:main:decay_rate}, that $\ds\int_t^{+\infty} \|u(s)-u_\infty\|^2\,ds\le C$. Thus, 
$$\|u(t)-u_\infty\|^2\le 2\int_t^{+\infty}u'(s)\cdot (u(s)-u_\infty)\,ds\le 2\left(\int_t^{+\infty} \|u(s)-u_\infty\|^2\,ds\right)^{\sfrac12}\xi(t)^{\sfrac12},$$ 
which gives the first part of \eqref{e:main:decay_rate2}. 
\end{proof}	
\begin{oss}\label{oss:main:loja}
We notice that in Proposition \ref{p:decay} it is sufficient to suppose that the \L ojasiewicz inequality \eqref{e:main:loja} holds for the time-slices of the flow $u(t)$, $t\ge 0$.
\end{oss}

To conclude this section we state a simple corollary of Proposition \ref{p:decay}, whose proof is left to the reader.

\begin{coro}\label{c:1} With the same notations of Proposition \ref{p:decay}, let $\varphi\in \mathcal S$ be the unique minimizer of $\mathcal F$ and suppose that
	
	(a) the flow \eqref{e:main:flow1K} of $-\nabla \mathcal F$ in $\mathcal K$ depends continuously on the initial datum; 

(b) $\cF$ has the constrained \L ojasiewicz property on $\mathcal S_\varphi=\{\varphi\}$ with $E_L=\delta_L=+\infty$ and $\gamma=\sfrac12$.

\noindent	Then, any global solution $u:\,]0,+\infty[\,\to \mathcal K$ converges to $\varphi$ and there is a constant $C>0$ such that
	\begin{equation}\label{e:main:decay_rate_1}
	\|u(t)-\varphi\|\le  e^{-Ct}\qquad\text{and}\qquad \mathcal F(u(t))-\mathcal F(\varphi)\le e^{-Ct},
	\end{equation}
	for every $t\ge 1$. 
\end{coro}

\section{Logarithmic epiperimetric inequalities}\label{s:epi}

In this section, we show that if a functional $\G$ satisfies a suitable slicing lemma, and the slicing functional $\cF$ is of the types considered in the previous section, then with a very general computation we can deduce that the so-called logarithmic epiperimetric inequality holds for $\G$. The subtlety here is that the competitor is going to live in a constrained subset of the domain of the functional which is not analytic. The link with the previous sections depends on the fact that we will use a parabolic inequality to define such a competitor. 

In this section we fix $\mathcal H=L^2(\de B_1)$, $\mathcal K\subset \mathcal H$ a convex cone, $\mathcal W=H^2(\de B_1)$.

\begin{prop}\label{p:epiK}
	Let $\G$ be a functional satisfying the following properties.
	
	(SL) There exist a constant $k\in \N$ and a functional $\cF\colon H^1(\de B_1) \to \R $  such that for every $u: [0,1]\times \de B_1\ni (r,\theta)\mapsto u(r,\theta)\in\R$, with $u\in H^1([0,1]\times \de B_1)$, the following slicing inequality holds  
		\begin{equation}\label{e:slicing}
		\G(r^k\,u)\leq  \int_0^1 \cF(u(r,\cdot))\, r^{2k+d-3}\, dr+C_{\text{\tiny\sc sl}}\, \int_{0}^1\int_{\de B_1} |\de_r u|^2\,r^{2k+d-1}\,d\HH^{d-1}dr\,,
		\end{equation} 
		for a geometric constant $C_{\text{\tiny\sc sl}}>0$, with equality if and only if $u$ is $0$-homogeneous (constant in the first variable).

		(FL) There is an open set $\mathcal U_\text{\tiny\sc fl}\subset \mathcal K$ and a constant $\eps_\text{\tiny\sc fl}>0$ such that for every $u_0\in\mathcal U_\text{\tiny\sc fl}$ there exists a strong solution $u\in H^1\big(]0,\eps_\text{\tiny\sc fl}[\,, \mathcal H\big)\cap L^2\big(]0,\eps_\text{\tiny\sc fl}[\,,\mathcal W\cap  \mathcal K\big)$ of
		\begin{equation}\label{e:flow1K}
		\begin{cases}
		\big(u'(t)+\nabla \cF(u(t))\big)\cdot (v-u(t))\ge0\quad\text{for every}\quad v\in \mathcal K \,,\quad 0< t\leq \eps_\text{\tiny\sc fl} \\
		u(0)=u_0\,,
		\end{cases}
		\end{equation}
		and the flow is continuous with respect to the initial datum (see Definition \ref{def:main:continuity}).
		
		(\L S) $\cF$ has the constrained \L ojasiewicz property (Definition \ref{def:main:loja}) with respect to $\mathcal S=\{\psi\}$, where $\psi\in \mathcal W\cap\mathcal K$ is a critical point of $\cF$ in $\mathcal K$ (see \eqref{e:main:critical}).

	Under the conditions (SL), (FL) and (\L S), there are constants $\delta_0>0$ and $E>0$, depending only on the dimension and $\psi$, such that: if $c\in H^1(\partial B_1)\cap\mathcal K$ satisfies  
	$$\|c-\psi\|_{L^2(\partial B_1)}\leq \delta_0\qquad\text{and}\qquad \mathcal F(c)-\mathcal F(\psi)\le E,$$ 
	then there exists a function $h=h(r,\theta)\in H^1(B_1)$ satisfying $h(r,\cdot)\in\mathcal K$, for every $r\in(0,1]$, and 
	\begin{equation}\label{e:logepiK}
	\G(h)-\G(\phi)\leq \left(1-\eps |\G(z)-\G(\phi)|^{1-2\gamma} \right) (\G(z)-\G(\phi))
	\end{equation}
	where $\phi(r,\theta):=r^k\psi(\theta)$, $z(r,\theta):=r^k c(\theta)$, $\eps>0$ is a universal constant and $\gamma>0$ is the exponent from  (\L S). 
\end{prop}

\begin{proof}
	Notice that if $\G(z)-\G(\phi)\leq 0$, then choosing $h:=z$ trivially gives the inequality. Therefore we can assume that 
	$$0<\G(z)-\G(\phi)=\frac{1}{2k+d-2}\big(\cF(c)-\cF(\psi)\big).$$
	We construct $h$ in the following way. Let $u:\,]0,\alpha[\,\to \mathcal K$ be a strong solution of
	\begin{equation}\label{e:main:flow1}
	\begin{cases}
	\big(u'(t)+\nabla \mathcal F(u(t))\big)\cdot (v-u(t))\ge0\quad\text{for every}\quad v\in \mathcal K\,,\\
	u(0)=u_0,
	\end{cases} 
	\end{equation}
	where $\alpha\leq\min\{\eps_\text{\tiny\sc fl},\eps_2\}$ and $\eps_2$ is chosen so that
	\begin{equation}\label{e:eps_2}
	\cF(u_0)-\cF(\psi)\leq 2 \big(\cF(u(t))-\cF(\psi)\big)\quad\text{for every} \quad 0< t \leq \eps_2\,,
	\end{equation}
	with equality exactly at $\eps_2$ (notice that $\eps_2$ is well-defined since $t\mapsto \cF(u(t))$ is continuous and non-increasing in $t$). We then extend $u$ to be constant on $[\alpha,+\infty[\,$, that is, $u(t,\cdot)\equiv u(\alpha,\cdot)$ for every $t\geq \alpha$. Then we define the competitor $h$, in polar coordinates,  as
	\begin{equation}
	h(r,\theta):=u\left(-\alpha \log(r),\theta\right)\,.
	\end{equation}
	For the sake of simplicity, we set $\|\cdot\|_2:=\|\cdot\|_{L^2(\partial B_1)}$. By (SL) we have 
	\begin{align}
	\G(h) - &\G(z) 
	\stackrel{\text{\rm(SL)}}{=}   \int_0^1 \big(\cF(h(r,\cdot))-\cF(c)\big)\, r^{2k+d-3}\, dr+C_{\text{\tiny\sc sl}}\, \int_{0}^1\int_{\de B_1} |\de_r h|^2\,r^{2k+d-1}\,d\HH^{d-1}dr\notag\\
	&=\frac1\alpha\int_0^\infty \left(\cF(u(t))-\cF(u^0)\right) \,e^{-\frac{t(2k+d-2)}\alpha} dt+C_{\text{\tiny\sc sl}}\alpha\, \int_{0}^\infty\int_{\de B_1}  | u'|^2  d\HH^{d-1}e^{-\frac{t(2k+d-2)}\alpha}\, dt\notag\\
	&= \frac1\alpha \int_0^\infty \int_0^{\min\{t,\alpha\}} \nabla\cF(u(\tau)) \cdot u'(\tau) \,d\tau \,e^{-\frac{t(2k+d-2)}\alpha} dt
	+ C_{\text{\tiny\sc sl}}\alpha \int_0^{\alpha}  \|u'(t)\|_{2}^2\,e^{-t(2k+d-2)}\,dt\notag\\
	&=\frac1\alpha\int_0^\infty \nabla\cF(u(\tau)) \cdot u'(\tau)\, \int_{\max\{\tau,\alpha\}}^\infty e^{-\frac{t(2k+d-2)}\alpha} \,dt  \,d\tau
	+ C_{\text{\tiny\sc sl}}\alpha \int_0^{\alpha}  \|u'(t)\|_{2}^2\,e^{-\frac{t(2k+d-2)}\alpha}\,dt\notag\\
	&=- \frac1\alpha\int_0^{\alpha} \left(-\frac1{(2k+d-2)}\nabla\cF(u(t)) \cdot u'(t)  -C_{\text{\tiny\sc sl}} \alpha^2  \|u'(t)\|_2^2\right)\,e^{-\frac{t(2k+d-2)}\alpha}\,dt\,\notag\\
	&=- \frac1\alpha\int_0^{\alpha} \left(\frac1{(2k+d-2)}\|u'(t)\|_2^2  -C_{\text{\tiny\sc sl}} \alpha^2 \|u'(t)\|_2^2\right)\,e^{-\frac{t(2k+d-2)}\alpha}\,dt\,\notag\\
	&=- \frac1\alpha\int_0^{\alpha} C_{d,k}\|\nabla\cF(u(t))\|_2^2 \,e^{-\frac{t(2k+d-2)}\alpha}\,dt\,,\label{e:forza_26}
	\end{align}
   where the next to last equality is due to Lemma \ref{l:puppa1} (i) and the last equality is due to Lemma \ref{l:puppa1} (iii) and a choice of $\alpha>0$ small enough, depending only on $d,k$ and $C_{\text{\tiny\sc sl}}$.  
   
   We next fix $\bar\eps\in\,]0,1[$. Using the property (\L S), with $C_L$ being the constant from the constrained \L ojasiewicz inequality, and the previous computation we calculate 
	\begin{align*}
	(\G(h)&-\G(\phi)) -(1-\bar\eps)(\G(z)-\G(\phi))\notag\\
	&=  -\frac{C_{d,k}}\alpha \int_0^{\alpha}\|\nabla\cF(u(t))\|_2^2\, \,e^{-\frac{t(2k+d-2)}\alpha}\,dt+  \frac{\bar\eps}{2k+d-2} \big(\cF(u_0)-\cF(\psi)\big)\notag\\
	&\stackrel{\text{\rm(\L S)}}{\leq}- \frac{C_{d,k}}\alpha \int_0^{\alpha} C_{L}^2 \big(\cF(u(t))-\cF(\psi)\big)^{2-2\gamma} \,e^{-\frac{t(2k+d-2)}\alpha}\,dt+\frac{\bar\eps}{2k+d-2} \big(\cF(u_0)-\cF(\psi)\big)\notag\\
	&\stackrel{\eqref{e:eps_2}}{\leq} -\frac{C_{d,k}C_{L}^2 }\alpha \int_0^{\alpha} 2^{2-2\gamma} \big(\cF(u_0)-\cF(\psi)\big)^{2-2\gamma}\,e^{-\frac{t(2k+d-2)}\alpha}\,dt +\frac{\bar\eps}{2k+d-2} \big(\cF(u_0)-\cF(\psi)\big)\notag\\
	&\leq  -\frac1{2k+d-2} \Big( C_{d,k,\gamma}C_{L}^2 \big(1-e^{-(2k+d-2)}\big)-\eps  \Big)\big(\cF(u_0)-\cF(\psi)\big)^{2-2\gamma}<0
	\end{align*}
	where in the last inequality we chose $\bar\eps:=\eps \big(\cF(u_0)-\cF(\psi)\big)^{1-2\gamma}$ for some $\eps>0$ small enough depending on $d$, $k$, $\gamma$ and $C_{L}$. Notice that we are allowed to apply (\L S) to $u(t)$ for every $0<t\leq \alpha$, by choosing $\delta_0$ small enough (depending on the dimension and $\psi$) and using the continuity of the flow with respect to the initial datum.
\end{proof}

\section{Constrained \L ojasiewicz  inequalities for obstacle and thin-obstacle problems}\label{s:loja}

This section is dedicated to the proofs of the theorems stated in the introduction. In particular we will show that constrained \L ojasiewicz-type inequalities hold in all the problems considered there and then conclude using the abstract results of the previous sections.

\subsection{The obstacle problem in a ball} Let $B_1\subset \R^d$ be the unit ball and $g\in H^1(B_1)$ be a given non-negative function. Let $\mathcal H=L^2(B_1)$, $\mathcal W=H^2(B_1)$ and 
$$\mathcal K_{\text{\tiny\sc ob}}=\big\{u\in H^1(B_1)\ :\ u-g\in H^1_0(B_1),\ u\ge 0\ \text{in}\ B_1\big\}.$$
Recall that the obstacle energy is given by
$$
\mathcal F_{\text{\tiny\sc ob}}(u)=\frac12\int_{B_1}|\nabla u|^2\,dx+\int_{B_1}u\,dx\qquad\text{and}\qquad \nabla \mathcal F_{\text{\tiny\sc ob}}(u)=\Delta u-1.
$$
Let $\varphi\in \mathcal K_{\text{\tiny\sc ob}}$ be the unique solution of the obstacle problem
\begin{equation}\label{e:ost}
\min_{u\in\mathcal K_{\text{\tiny\sc ob}}}\mathcal F_{\text{\tiny\sc ob}}(u).
\end{equation}
Then, $\varphi\in H^2(B_1)$ is the only critical point of $\mathcal F_{\text{\tiny\sc ob}}$ in $\mathcal K_{\text{\tiny\sc ob}}$ (in the sense of \eqref{e:main:critical}) and is also the unique solution of the problem
$$
\Delta\varphi =\ind_{\{\varphi>0\}}\quad\text{in}\quad B_1,\qquad \varphi=g\quad\text{on}\quad \partial B_1.
$$

\begin{prop}[Constrained \L ojasiewicz for the obstacle problem]\label{p:ost:loja}
	Let $\mathcal K_{\text{\tiny\sc ob}}$ and $\mathcal F_{\text{\tiny\sc ob}}$ be as above and $\varphi$ be the solution of the obstacle problem \eqref{e:ost}. Then, there is a dimensional constant $C_d>0$ such that 
	\begin{equation}\label{e:ost:loja}
	\big(\mathcal F_{\text{\tiny\sc ob}}(u)-\mathcal F_{\text{\tiny\sc ob}}(\varphi)\big)^{\sfrac12} \leq C_d\, \|\nabla \mathcal F_{\text{\tiny\sc ob}}(u)\|_{\mathcal K_{\text{\tiny\sc ob}}}\quad\text{for every}\quad u\in H^2(B_1)\cap \mathcal K_{\text{\tiny\sc ob}}.
	\end{equation}
	\end{prop}
	
\begin{proof}To simplify the notations we drop the index ${\text{\sc ob}}$.
Let $u\in H^2(B_1)\cap \mathcal K$. Then, we have 
\begin{align*}
\left\|\Delta u-1\right\|_{\mathcal K}
&\geq- \frac{1}{\|u-\varphi\|_{L^2}}\int_{B_1} (u-\varphi)\left(\Delta u-1\right)\,dx\\
&=- \frac{1}{\|u-\varphi\|_{L^2}}\int_{B_1} (u-\varphi)\Delta (u-\varphi) \,dx+\frac{1}{\|u-\varphi\|_{L^2}} \int_{B_1\cap\{\varphi=0\}} (u-\varphi)\,dx\\
&\ge \frac{1}{\|u-\varphi\|_{L^2}}\left(\frac12\int_{B_1}  |\nabla (u- \varphi)|^2\,dx+ \int_{B_1\cap\{\varphi=0\}} (u-\varphi)\,dx\right)\,.
\end{align*}
Next since
\begin{align}
\int_{B_1} |\nabla(u-\varphi)|^2\,dx&=\int_{B_1}\Big(|\nabla u|^2-|\nabla \varphi|^2+2 \nabla \varphi \cdot\nabla (\varphi-u)\Big)\,dx\notag\\
& =\int_{B_1}|\nabla u|^2\,dx-\int_{B_1}|\nabla \varphi|^2\,dx+2\int_{B_1\cap \{\varphi>0\}}(u-\varphi)\,dx,\label{e:ost:grad}
\end{align}
we conclude that 
\begin{align}\label{e:???}
\left\|\Delta u-1\right\|_{L^2}\ge \frac{1}{\|u-\varphi\|_{L^2}}\big(\mathcal F(u)-\mathcal F(\varphi)\big).
\end{align}
Let $C_d$ be the constant of the Poincar\'e inequality for $u-\varphi\in H^1_0(B_1)$. Using again \eqref{e:ost:grad} and the fact that $u\ge 0$ in $B_1$, we have 
$$
\|u-\varphi\|_{L^2}^2\leq C_d \|\nabla(u-\varphi)\|_{L^2}^2\leq 2C_d \big(\mathcal F(u)-\mathcal F(\varphi)\big)\,.
$$
This, together with \eqref{e:???} gives \eqref{e:ost:loja}.
\end{proof}

\subsection{The thin-obstacle problem in a ball}
Let $B_1$ be the unit ball in $\R^d$, $B_1^+:=B_1\cap \{x_d>0\}$ and $B_1'=B_1\cap \{x_d=0\}$; we will use the notation $x=(x',x_d)\in \R^{d-1}\times\R$. Let $g\in H^1(B_1^+)$ be such that $g\ge 0$ on $B_1'$. Let $\mathcal H=L^2=L^2(B_1^+)$, $\mathcal W=H^2(B_1^+)$ and  
$$
\mathcal K_{\text{\tiny\sc th}}:=\Big\{u\in H^1(B_1^+)\ :\ u\ge 0\ \text{ on }\ B_1',\ u=g\ \text{ on }\ \partial B_1\cap\{x_d>0\}\Big\}.
$$
Recall that the thin-obstacle energy is given by
$$
\mathcal F_{\text{\tiny\sc th}}(u)=\frac12\int_{B_1}|\nabla u|^2\,dx\qquad\text{with}\qquad \nabla \mathcal F_{\text{\tiny\sc th}}(u)=\Delta u.
$$
Let $\varphi$ be the unique solution of the thin-obstacle problem 
\begin{equation}\label{e:sottile}
\min_{u\in\mathcal K_{\text{\tiny\sc th}}}\mathcal F_{\text{\tiny\sc th}}(u)\,,
\end{equation}
where all functions in $\mathcal K_{\text{\tiny\sc th}}$ are extended even to the full ball.

\begin{prop}[Constrained \L ojasiewicz for the thin-obstacle problem]\label{p:sottile}
	Let $\mathcal K_{\text{\tiny\sc th}}$, $\mathcal F_{\text{\tiny\sc th}}$ and $\varphi$ be as above. 
There is a dimensional constant $C_d$ such that 
	\begin{equation}\label{e:sottile:loja}
\big(\mathcal F_{\text{\tiny\sc th}}(u)-\mathcal F_{\text{\tiny\sc th}}(\varphi)\big)^{\sfrac12} \leq C_d \|\nabla \mathcal F_{\text{\tiny\sc th}}(u)\|_{\mathcal K},
\end{equation}
for every $u\in H^2(B_1^+)\cap \mathcal K_{\text{\tiny\sc th}}$ such that
	\begin{equation}\label{e:sottile:conditions}
	\frac{\partial u}{\partial x_d}\le 0\quad \text{on}\quad B_1'\qquad \text{and}\qquad u\frac{\partial u}{\partial x_d}= 0\quad \text{on}\quad B_1'.
	\end{equation}
\end{prop}

\begin{proof} For simplicity of notations we drop the index ${\text{\sc th}}$.
Let $u\in H^2(B_1^+)\cap \mathcal K$ be a function satisfying \eqref{e:sottile:conditions}. 
By definition of $\|\cdot \|_\mathcal K$ (see \eqref{e:main:Knorm}), we have the estimate  
\begin{align*}
\|\nabla \mathcal F(u)\|_{\mathcal K}
&\geq- \frac1{\|u-\varphi\|_{L^2}}\int_{B_1^+} (u-\varphi)\Delta u\,dx=- \frac1{\|u-\varphi\|_{L^2}}\int_{B_1^+} (u-\varphi)\Delta (u-\varphi)\,dx\\
&=\frac{1}{\|u-\varphi\|_{L^2}}\left(\int_{B_1^+}  |\nabla (u-\varphi)|^2\,dx-\int_{B_1'}  (u-\varphi)\frac{\partial (u-\varphi)}{\partial n}dx'\right),
\end{align*}
where $n$ is the exterior normal to $B_1^+$. 
On the other hand, we have
\begin{align}
\frac12\int_{B_1^+} |\nabla(u-\varphi)|^2\,dx&=\frac12\int_{B_1^+}|\nabla u|^2\,dx-\frac12\int_{B_1^+}|\nabla \varphi|^2\,dx+\int_{B_1^+}\nabla \varphi\cdot \nabla(\varphi-u)\,dx\notag\\
&=\mathcal F(u)-\mathcal F(\varphi)+\int_{B_1'} (\varphi-u)\frac{\partial \varphi}{\partial n}\,dx'.\label{e:sottile:grad}
\end{align}
Using that $\varphi\frac{\partial u}{\partial n}\ge 0$, $u\frac{\partial \varphi}{\partial n}\ge 0$ and  $u\frac{\partial u}{\partial n}=\varphi\frac{\partial \varphi}{\partial n}= 0$ on $B_1'$, we have
\begin{align*}
\|\nabla \mathcal F(u)\|_{\mathcal K}
&\ge\frac{1}{\|u-\varphi\|_{L^2}}\left(\frac12\int_{B_1^+}  |\nabla (u-\varphi)|^2\,dx+\int_{B_1'}  \left(\varphi\frac{\partial u}{\partial n}+u\frac{\partial \varphi}{\partial n}\right)dx'\right)\\
&=\frac{1}{\|u-\varphi\|_{L^2}}\left(\mathcal F(u)-\mathcal F(\varphi)+\int_{B_1'} \varphi\frac{\partial u}{\partial n}\,dx'\right)\ge\frac{\mathcal F(u)-\mathcal F(\varphi)}{\|u-\varphi\|_{L^2}}.
\end{align*}
On the other hand, the fact that $u-\varphi=0$ on $\partial B_1\cap\{x_d>0\}$, the Poincaré inequality and \eqref{e:sottile:grad} give that 
\begin{align*}
\|u-\varphi\|_{L^2(B_1^+)}^2\le C_d\int_{B_1^+} |\nabla(u-\varphi)|^2\,dx\le 2C_d\big(\mathcal F(u)-\mathcal F(\varphi)\big),
\end{align*}
which concludes the proof of \eqref{e:sottile:loja}.
\end{proof}

\subsection{Obstacle problem on a compact manifold}
Let $(\mathcal M,g)$ be a compact connected oriented Riemannian manifold of dimension $d\ge 2$. We denote by $\Delta$ and $\nabla$ the Laplace-Beltrami operator and the gradient on $\mathcal M$, respectively. We denote by $dV_g$ the volume form on $\mathcal M$, in local coordinates $dV_g = \det (g_{ij}) dx^1 \wedge\dots\wedge dx^n$. We will denote by $L^2(\mathcal M)$ the space of Lebesgue measurable square real integrable functions and, for $u\in L^2(\mathcal M)$, we will use the notation $\|u\|_2=\left(\int_{\mathcal M} u^2\,dV_g\right)^{\sfrac12}$. The associated scalar product in $L^2(\mathcal M)$ will be denoted by
$$u\cdot v=\langle u,v\rangle=\langle u,v\rangle_{L^2(\mathcal M)}=\int_{\mathcal M}uv\,dV_g\,,\quad\text{for}\quad u,v\in L^2(\mathcal M).$$ 
The Sobolev space $H^1(\mathcal M)$ on $(\mathcal M,g)$ is defined as the closure of the smooth functions on $\mathcal M$ with respect to the norm 
$$\|u\|_{H^1}=\int_{\mathcal M}\big(g(\nabla u,\nabla u)+u^2\big)\,dV_g.$$ 
Moreover, we will often use the notations $\nabla u\cdot\nabla v:=g(\nabla u, \nabla v)$, for the scalar product with respect to the metric $g$, and $|\nabla u|^2:=g(\nabla u, \nabla u)$, for the induced norm. The higher order Sobolev spaces $H^k(\mathcal M)$ are defined analogously. The Laplace-Beltrami operator $\Delta$ is defined on $H^2(\mathcal M)$ with values in $L^2(\mathcal M)$ and also by duality, as an operator $\Delta: H^1(\mathcal M)\to H^{-1}(\mathcal M)$; in both cases we will use the notation  
$$\int_{\mathcal M} (\Delta u)\, v\, dV_g:=-\int_{\mathcal M} \nabla u\cdot \nabla v\,dV_g\quad\text{for every}\quad u,v\in H^1(\mathcal M).$$
It is well known that the spectrum of the Laplace-Beltrami operator $\Delta$ is discrete and can be written as an increasing sequence of real positive eigenvalues $0=\lambda_1<\lambda_2\le \lambda_3\le \dots\le \lambda_k\le \dots$ counted with their multiplicity. The corresponding eigenfunctions $\phi_j\in H^1({\mathcal M})$, $j\in \N$, are smooth on ${\mathcal M}$ and form an infinite orthonormal basis of $L^2({\mathcal M})$, precisely, 
$$\ds \int_{\mathcal M} uv\, dV_g=\delta_{ij}\qquad\text{and}\qquad \int_{\mathcal M} \nabla u\cdot\nabla v\, dV_g=\lambda_i\delta_{ij}.$$
Let $\lambda\in\R$ be given and $\F_{\text{\tiny\sc ob}}^\lambda$ be the functional 
\begin{equation}
\label{e:ostacolo:F}
\F_{\text{\tiny\sc ob}}^\lambda(u) = \frac12\int_{{\mathcal M}}\left(|\nabla u|^2-\lambda u^2\right) dV_g+\int_{{\mathcal M}}u\, dV_g. 
\end{equation}
For $u,v\in H^1({\mathcal M})$, we will use the notation  
$$v\cdot \nabla \F_{\text{\tiny\sc ob}}^\lambda(u)=\nabla \F_{\text{\tiny\sc ob}}^\lambda(u)[v]=\delta \F_{\text{\tiny\sc ob}}^\lambda(u)[v]=\lim_{t\to 0}\frac1t\left(\mathcal F_{\text{\tiny\sc ob}}^\lambda(u+tv)-\mathcal F_{\text{\tiny\sc ob}}^\lambda(u)\right),$$
and we notice that 
$$\nabla \mathcal F_{\text{\tiny\sc ob}}^\lambda(u)=-\Delta u - \lambda u + 1\quad\text{for every}\quad u\in H^2(\mathcal M)$$
Let $\mathcal K_{\text{\tiny\sc ob}}^{\mathcal M}:=\{u\in L^2({\mathcal M})\,:\, u\geq 0\}$. For $u\in\mathcal K_{\text{\tiny\sc ob}}^{\mathcal M}\cap H^2(\mathcal M)$, we define  $\|\nabla \mathcal F_{\text{\tiny\sc ob}}^\lambda(u)\|_{\mathcal K_{\text{\tiny\sc ob}}^{\mathcal M}}$ as in \eqref{e:main:Knorm}. 
\medskip

Let $\lambda$ be an eigenvalue of the Laplace-Beltrami operator $\Delta$ on ${\mathcal M}$. Then all the critical points of $\mathcal F_{\text{\tiny\sc ob}}^\lambda$ are 
of the form $\lambda^{-1}+\phi$, for some $\lambda$-eigenfunction $\phi$. We denote by $\mathcal S_\lambda$ the set of all non-negative critical points of the functional $\mathcal F_{\text{\tiny\sc ob}}^\lambda$. Notice that $\mathcal F_{\text{\tiny\sc ob}}^\lambda$ is constant on $\mathcal S_\lambda$, that is, 
$$
\mathcal F_{\text{\tiny\sc ob}}^\lambda(\mathcal S_\lambda)=\mathcal F_{\text{\tiny\sc ob}}^\lambda(\varphi)\,,\qquad\forall\varphi\in \mathcal S_\lambda.
$$

\begin{oss}
The set $\mathcal S_\lambda$ is bounded both in $L^2(\mathcal M)$ and $H^1(\mathcal M)$. Indeed, suppose that this is not the case. Then there is a sequence of $\lambda$-eigenfunctions $\phi_n$ such that $\|\phi_n\|_2=1$ and a sequence $C_n\to\infty$ such that $\lambda^{-1}+C_n\phi_n\in\mathcal S_\lambda$. But then we have also that $\psi_n:=C_n^{-1}\lambda^{-1}+\phi_n\in\mathcal S_\lambda$. Now, since $\phi_n$ is bounded in $H^1(\mathcal M)$, up to a subsequence, $\phi_n$ and $\psi_n$ converge strongly in $L^2(\mathcal M)$ to a function $\phi\in H^1(\mathcal M)$, such that $\phi\ge 0$ on $\mathcal M$, $\int\phi^2\,dV_g=1$ and $\phi$ is a $\lambda$-eigenfunction, which is a contradiction.
\end{oss}

We next prove a \L ojasiewicz inequality in a neighborhood of the family of critical points $\mathcal S_\lambda$. For any $u\in L^2({\mathcal M})$, we set  
$$\text{dist}_{2}(u,\mathcal S_\lambda):=\inf\big\{\|u-\varphi\|_2\ :\ \varphi\in \mathcal S_\lambda\big\},$$
and, for any $\gamma\in(0,\sfrac12]$, we define the function $f_\gamma:\R\to\R^+$ as 
\begin{equation}\label{e:f_gamma}
f_\gamma(t):=\begin{cases}
t^{\sfrac12},\ \text{if}\ t\ge 1,\\
t^{1-\gamma},\ \text{if}\ 0\le t\le 1,\\
0, \ \text{if}\ t\le 0,\\
\end{cases}
\end{equation}

\begin{prop}[Constraint \L ojasiewicz inequality for the obstacle on a manifold]\label{p:ostacolo:loja}
	Let $\lambda>0$ be an eigenfunction of $\Delta$ on $\mathcal M$, and let $\mathcal F_{\text{\tiny\sc ob}}^\lambda$, $\mathcal S_\lambda$ be as above. Then, there are constants $C, \delta>0$ (depending on $(\mathcal M,g)$ and $\lambda$) and $\gamma \in (0,\sfrac12)$ (depending only on the dimension $d=\dim\mathcal M$) such that: 
	\begin{equation}
	\label{e:ostacolo:loja}
	f_\gamma\big(\F_{\text{\tiny\sc ob}}^\lambda(u) - \F_{\text{\tiny\sc ob}}^\lambda(\mathcal S_\lambda)\big) \leq C \|\nabla \F_{\text{\tiny\sc ob}}^\lambda (u)\|_{\mathcal K},
	\end{equation}
 for every	$u\in H^2(\mathcal M)\cap\mathcal K_{\text{\tiny\sc ob}}^{\mathcal M}$ such that $\text{\rm dist}_{2}(u,\mathcal S_\lambda)\le \delta$.
	In particular, for every $E\ge 1$ and every $u\in H^2(\mathcal M)\cap\mathcal K_{\text{\tiny\sc ob}}^{\mathcal M}$ satisfying  $\text{\rm dist}_{2}(u,\mathcal S_\lambda)\le \delta$ and $\mathcal F_{\text{\tiny\sc ob}}^\lambda(u)-\mathcal F_{\text{\tiny\sc ob}}^\lambda(\mathcal S_\lambda)\le E$, we have 
	\begin{equation}
	\label{e:ostacolo:loja2}
	\big(\F_{\text{\tiny\sc ob}}^\lambda(u) - \F_{\text{\tiny\sc ob}}^\lambda(\mathcal S_\lambda)\big)_+^{1-\gamma} \leq CE^{\frac12-\gamma} \|\nabla \F_{\text{\tiny\sc ob}}^\lambda (u)\|_{\mathcal K}.
	\end{equation}
\end{prop}

\begin{proof}
	For the sake of simplicity we set $\mathcal F=\mathcal F_{\text{\tiny\sc ob}}^\lambda$, $\mathcal K=\mathcal K_{\text{\tiny\sc ob}}^{\text{\tiny\sc m}}$ and $\mathcal S=\mathcal S_\lambda$. Let $\varphi\in\mathcal S$ be such that $\|u-\varphi\|_2\le 2\delta$. Notice that $u-\varphi$ can be uniquely decomposed in Fourier series as $u-\varphi=Q_-+Q_0+\eta$, where $Q_-$ contains only lower eigenmodes (corresponding to eigenvalues $<\lambda$), $Q_0$ is a $\lambda$-eigenfunction and $$\ds\eta(x)=\sum_{\{j:\lambda_j>\lambda\}}c_j\phi_j(x),$$ 
which contains only higher eigenmodes (corresponding to eigenvalues $>\lambda$).  Thus, $u=Q_-+Q_0+\varphi+\eta$ and $\|Q_-\|_2, \|Q_0\|_2, \|\eta\|_2\le 2\delta$. We now consider $M:=\max_{x\in {\mathcal M}}\{-Q_-(x)-Q_0(x)-\varphi(x)\}$ and suppose that the maximum is realized in a point $x_M\in\mathcal M$. Notice that since $Q_-+Q_0$ is a finite sum of (smooth) eigenfunctions, there is a constant $C>0$ (depending on $\mathcal M$ and $\lambda$) such that $\|Q_-+Q_0\|_{L^\infty}\le C\delta$. Thus, if $M>0$, then $x_M\in\{\varphi<C\delta\}$ and $M\le C\delta$. We now choose $\delta$ such that $10C\delta< c_\lambda:=\lambda^{-\sfrac12}$  and we claim that the function
	$$\tilde u=Q_-+Q_0+\varphi+\frac{2M}{c_\lambda}\big(c_\lambda-\varphi\big)$$
	is non-negative. Indeed, it is sufficient to consider the following two cases:
	
	$\bullet$ on the set $\{\varphi\ge 2C\delta\}$, we have that  
	$$\tilde u=\left(Q_-+Q_0+\frac12\varphi\right)+2M+\varphi\left(\frac12-\frac{2M}{c_\lambda}\right)\ge 0,$$
	since each of the three terms is non-negative;
	
	$\bullet$ on the set $\{\varphi\le 2C\delta\}$, we have that
	$$\tilde u\ge Q_-+Q_0+\varphi+\frac{2M}{c_\lambda}\big(c_\lambda-2C\delta\big)\ge   Q_-+Q_0+\varphi+M\ge 0.$$

Next, using the fact that $c_\lambda-\varphi$ is a $\lambda$-eigenfunction (notice that the integral of $c_\lambda-\varphi$ on $\mathcal M$ vanishes, due to the fact that $\lambda>0$), we calculate 
	\begin{equation*}
	\begin{split}
	-(\tilde u-u)\cdot\nabla \F(u)
	&= \int_{\mathcal M}\big(-\Delta  u - \lambda  u + 1\big) \left(\eta-\frac{2M}{c_\lambda}\left(c_\lambda-\varphi\right)\right) dV_g\\
	&= \int_{\mathcal M}\big(-\Delta  u - \lambda  u + 1\big) \eta\, dV_g\\
	&= \int_{\mathcal M}\big(-\Delta  (Q_-+Q_0+\eta) - \lambda  (Q_-+Q_0+\eta)\big) \eta\, dV_g\\
	&= \int_{\mathcal M} \left(|\nabla \eta|^2 -\lambda \eta^2\right)dV_g = \sum_{j:\lambda_j>\lambda} c_j^2(\lambda_j-\lambda)=2\mathcal F(\eta).
	\end{split}
	\end{equation*}
		Notice that since the set of eigenvalues is discrete, there is a (spectral gap) constant $G(\lambda)>0$ such that $\lambda_j-\lambda\ge G(\lambda)$, whenever $\lambda_j-\lambda>0$. In particular, we have the inequality 
		$$2\mathcal F(\eta)= \sum_{j:\lambda_j>\lambda} c_j^2(\lambda_j-\lambda)\ge G(\lambda)\sum_{j:\lambda_j>\lambda} c_j^2=G(\lambda)\|\eta\|_2^2.$$
		Thus, we get
	\begin{equation}\label{e:ostacolo:nablaF}
	\begin{split}
	\|\nabla \F(u)\|_{\mathcal K}& \ge \frac{-(\tilde u-u)\cdot\nabla \F(u)}{\|u-\tilde u\|_{2}} \ge\frac{2\mathcal F(\eta)}{\frac{2M}{c_\lambda}\|c_\lambda-\varphi\|_{2}+\|\eta\|_{2}} \\
	&\ge \frac{2\mathcal F(\eta)}{M\frac{2}{c_\lambda}(\|c_\lambda\|_2+\|\varphi\|_{2})+(G(\lambda)\mathcal F(\eta))^{\sfrac12}}\ge C\frac{\mathcal F(\eta)}{M+\mathcal F(\eta)^{\sfrac12}},\end{split}
	\end{equation}
where $C$ is a constant depending only on $\lambda$ and $\mathcal M$.
	On the other hand, we have 	
	\begin{align}\label{e:ostacolo:F-F}
	2\big(\F(u)- \F(\varphi)\big) &= \int_{\mathcal M} \big(|\nabla(u-\varphi)|^2- \lambda (u-\varphi)^2\big)dV_g\le \int_{\mathcal M} \big(|\nabla\eta|^2- \lambda \eta^2\big)dV_g=2\mathcal F(\eta).
	\end{align}
Now, in order to get \eqref{e:ostacolo:loja}, it only remains to estimate $M$ and put together \eqref{e:ostacolo:nablaF} and \eqref{e:ostacolo:F-F}. We notice that: 

$\bullet$ $Q_-+Q_0+\varphi$ is a (finite) linear combination of (orthonormal and smooth) eigenfunctions corresponding to eigenvalues $\le\lambda$; 

$\bullet$ the $L^2$ norm of $Q_-+Q_0+\varphi$ is bounded by a universal constant. 

\noindent As a consequence, there is a universal (Lipschitz) constant $L$, depending only on $\lambda$ and $\mathcal M$, such that $\|\nabla (Q_-+Q_0+\varphi)\|_{L^\infty(\mathcal M)}\le L$.  
\noindent Thus, since the negative part $\psi:=-\inf\{(Q_-+Q_0+\varphi),0\}$ is such that $\sup\psi=M$ is small enough (bounded by a constant depending on $\mathcal M$ and $\lambda$, as already mentioned above), we get that there is a constant $C$ (depending on $\mathcal M$) such that 
$$\|\psi\|_{L^2(\mathcal M)}^2\ge C L^{-d}M^{d+2}=C L^{-d}\|\psi\|_{L^\infty(\mathcal M)}^{d+2}.$$
Since $u\ge 0$ on $\mathcal M$, we have that $\psi\le \eta$ and so, 
$$M^{d+2}\le C^{-1}L^d\|\eta\|_2^2\le \frac{2L^d}{CG(\lambda)}\mathcal F(\eta),$$
which, together with \eqref{e:ostacolo:nablaF} and \eqref{e:ostacolo:F-F}, we get \eqref{e:ostacolo:loja} with $\ds\gamma=\frac{1}{d+2}$. 
\end{proof}

\subsection{The thin-obstacle problem on the sphere}
Let $\mathcal K_{\text{\tiny\sc th}}^\mS$ be the set of functions on the sphere which are non-negative on the equator $\{x_d=0\}\cap \partial B_1$. 
Let $m\in\N$, $\lambda:=\lambda(2m)=2m(2m+d-2)$ and $\mathcal F_{\text{\tiny\sc th}}^\lambda$ be the functional
$$\mathcal F_{\text{\tiny\sc th}}^\lambda(u)=\frac12\int_{\partial B_1}\big(|\nabla u|^2-\lambda u^2\big)\,d\HH^{d-1}.$$
Thus, $\nabla\mathcal F_{\text{\tiny\sc th}}^\lambda(u)=-\Delta u-\lambda u,$
where $\Delta$ is the Laplace-Beltrami operator on the sphere.
Notice that the operator $\nabla \mathcal F_{\text{\tiny\sc th}}^\lambda:H^2(\partial B_1)\to L^2(\partial B_1)$ is symmetric, that is, for any $u,v\in H^2(\partial B_1)$ we have  
$$v\cdot\nabla\mathcal F_{\text{\tiny\sc th}}^\lambda(u)=\int_{\partial B_1} v(-\Delta u-\lambda u)\,d\HH^{d-1}=\int_{\partial B_1} u(-\Delta v-\lambda v)\,d\HH^{d-1}=u\cdot\nabla\mathcal F_{\text{\tiny\sc th}}^\lambda(v).$$
The set of critical points of $\mathcal F_{\text{\tiny\sc th}}^\lambda$ is precisely the set of the eigenfunctions of the spherical Laplacian, corresponding to the eigenvalue $\lambda(k)$.
Let $\mathcal S_\lambda$ be the set of critical points of $\mathcal F_{\text{\tiny\sc th}}^\lambda$, which are positive on $\{x_d=0\}\cap \partial B_1$.  Notice that $\mathcal F_{\text{\tiny\sc th}}^\lambda$ vanishes on $\mathcal S_\lambda$, that is, $\mathcal F_{\text{\tiny\sc th}}^\lambda(\mathcal S_\lambda)\equiv0$.

\begin{prop}[Constrained \L ojasiewicz inequality for the thin-obstacle on the sphere]\label{prop:thin:loja}
Let $\mathcal K_{\text{\tiny\sc th}}^\mS$, $\mathcal F_{\text{\tiny\sc th}}^\lambda$, $\mathcal S_\lambda$ and $\lambda=\lambda(2m)$ be as above.
Then, there are constants $C$ and $\delta$, depending only on the dimension $d$ and the homogeneity $2m$, such that 
\begin{equation}\label{e:thin:loja}
f_\gamma\big(\mathcal F_{\text{\tiny\sc th}}^\lambda(u)\big)\le C\|\nabla\mathcal F_{\text{\tiny\sc th}}^\lambda(u)\|_{\mathcal K_{\text{\tiny\sc th}}^\mS}\quad\text{for every}\quad u\in\mathcal K_{\text{\tiny\sc th}}^\mS\cap H^2(\partial B_1)\quad\text{such that}\quad\text{\rm dist}_2(u,\mathcal S_\lambda)\le \delta,
\end{equation}
where $\gamma=\sfrac1d$, $f_\gamma$ is the function defined in  \eqref{e:f_gamma} and $\text{\rm dist}_2(u,\mathcal S_\lambda)$ is the $L^2(\partial B_1)$ distance from $u$ to the set $\mathcal S_\lambda$. In particular, for every $E\ge 1$, we have   
\begin{equation}
\label{e:thin:loja2s}
\big(\F_{\text{\tiny\sc th}}^\lambda(u)\big)_+^{1-\gamma} \leq CE^{\frac12-\gamma} \|\nabla \F_{\text{\tiny\sc th}}^\lambda (u)\|_{\mathcal K},
\end{equation}
for every $u\in H^2(\partial B_1)\cap\mathcal K_{\text{\tiny\sc th}}^\mS$ such that  $\text{\rm dist}_{2}(u,\mathcal S_\lambda)\le \delta$ and $\mathcal F_{\text{\tiny\sc th}}^\lambda(u)\le E$.
\end{prop}

\begin{proof}For the sake of simplicity, we set $\mathcal F = \mathcal F_{\text{\tiny\sc th}}^\lambda$,   $\mathcal K = \mathcal K_{\text{\tiny\sc th}}^\mS$ and $\mathcal S=\mathcal S_\lambda$.
Let $u\in \mathcal K$ be such that $\mathcal F(u)>\mathcal F(\mathcal S)=0$. By the definition of $\|\nabla\mathcal F\|_{\mathcal K}$, it is sufficient to prove that there is a function $\tilde u\in \mathcal K$ such that, for universal constants $C>0$ and $\gamma\in(0,\sfrac12]$ we have
\begin{equation}\label{e:thin:loja2}
\frac{-(\tilde u-u)\cdot \nabla \mathcal F(u)}{\|\tilde u-u\|_{2}}\ge C f_\gamma(\mathcal F(u)).
\end{equation}

\noindent We notice that $u$ can be decomposed in Fourier series on $\partial B_1$, using the eigenfunctions of the spherical Laplacian.  We write $u$ as $u=v_-+v_0+v_+$ and we set $\tilde u=v_-+\tilde v_0$, where:

$\bullet$  $v_-$ is the projection of $u$ on the space of eigenfunctions corresponding to eigenvalues $< \lambda(2m)$; 

$\bullet$  $v_+$ is the projection of $u$ on the space of eigenfunctions corresponding to eigenvalues $> \lambda(2m)$; 

$\bullet$ $v_0$ and $\tilde v_0$ are critical points for $\mathcal F$ (eigenfunctions for $\lambda(2m)$);  $\tilde v_0$ will be chosen later.
\medskip

\noindent We first calculate the left-hand side of \eqref{e:thin:loja2}. For the $L^2$ norm we have
\begin{equation*}
\|\tilde u-u\|_2=\left(\|\tilde v_0-v_0\|_2^2+\|v_+\|_2^2\right)^{\sfrac12}.
\end{equation*}
For the scalar product, we integrate by parts and use the orthogonality of the eigenfunctions:
\begin{align*}
-(\tilde u-u)\cdot\nabla\mathcal F(u)&=(v_0-\tilde v_0+v_+)\cdot\nabla\mathcal F(u)=u\cdot \nabla\mathcal F(v_0-\tilde v_0+v_+)\\
	&= u\cdot \nabla\mathcal F(v_+)=v_+\cdot \nabla\mathcal F(u)=v_+\cdot \nabla\mathcal F(v_+)=\mathcal F(v_+).
\end{align*}
Thus, we obtain
\begin{equation}\label{e:thin:key1}
-\frac{\tilde u-u}{\|\tilde u-u\|_{2}}\cdot \nabla \mathcal F(u)=\frac{\mathcal F(v_+)}{\left(\|\tilde v_0-v_0\|_2^2+\|v_+\|_2^2\right)^{\sfrac12}}.
\end{equation}
We now aim to estimate $\|\tilde v_0-v_0\|_2^2+\|v_+\|_2^2$ by $\mathcal F(v_+)$. First, we notice that
\begin{equation}\label{e:thin:key2}
\mathcal F(v_+)=\sum_{j=2m+1}^\infty c_j^2\big(\lambda(j)-\lambda(2m)\big)\ge \big(\lambda(2m+1)-\lambda(2m)\big)\sum_{j=2m+1}^\infty c_j^2= (4m+d-1) \|v_+\|_{2}^2,
\end{equation}
which estimates the second term. In order to give a bound for $\|\tilde v_0-v_0\|_2$ we will need to choose $\tilde v_0$ carefully. 
Let $\psi$ be an eigenfunction of the spherical Laplacian, corresponding to the eigenvalue $\lambda(2m)$, and such that $\psi=1$ on $\{x_d=0\}\cap\partial B_1$. We choose 
$$\tilde v_0=M\psi +v_0\qquad\text{where}\qquad M=-\inf_{\{x_d=0\}\cap \partial B_1}(v_-+v_0).$$
We now claim that there is a constant $C_{d,m}>0$, depending on $d$ and $m$, such that 
\begin{equation}\label{e:gamma}
M^d\le C_{d,m}\mathcal F(v_+).
\end{equation}
First of all, we notice that there is a constant $L_m$, depending only on $d$ and $m$, such that all the eigenfunctions corresponding to eigenvalues $\le \lambda(2m)$ are globally $L_m$-Lipschitz continuous on $\partial B_1$. In particular, the function $v_-+v_0$ is $L$-Lipschitz continuous for  $L=L_m(\|v_-\|_2^2+\|v_0\|_2^2)^{\sfrac12}$. Since $v_-+v_0+v_+$ is non-negative on $\mathbb{S}^{d-2}$ we get that 
\begin{equation*}
\int_{\mathbb{S}^{d-2}} v_+^2\,d\HH^{d-2}\ge\int_{\mathbb{S}^{d-2}} \big(\min\{0,v_-+v_0\}\big)^2\,d\HH^{d-2}\ge C_d M^2 \left(\frac{M}{L}\right)^{d-2}=\frac{C_d}{L^{d-2}}M^{d},
\end{equation*}
for some dimensional constant $C_d$. 
Now, by the trace inequality on the sphere $\partial B_1$ and the fact that the expansion of $v_+$ contains only eigenfunctions corresponding to frequencies higher than $\lambda(2m)$, we get
\begin{equation*}
\ds\int_{\mathbb{S}^{d-2}} v_+^2\,d\HH^{d-2}\ds\le C_d\int_{\partial B_1} \left(|\nabla_\theta v_+|^2+v_+^2\right)\,d\HH^{d-1}\le C_{d,m} \mathcal F(v_+),
\end{equation*}
which concludes the proof of \eqref{e:gamma}. Finally, \eqref{e:gamma} and \eqref{e:thin:key2} give  
$$\qquad\|\nabla\mathcal F(u)\|_{\mathcal K}\geq \frac{\mathcal F(v_+)}{\left(M^2\|\psi\|_2^2+\|v_+\|_2^2\right)^{\sfrac12}}\ge C_{d,m}  f_\gamma\big(\mathcal F(v_+)\big)\ge C_{d,m}  f_\gamma\big(\mathcal F(u)\big).
\qquad\qedhere
$$

\section{Proof of the main results}\label{s:proofs}

\subsection{Proof of Theorem \ref{t:flow1}} 
We first consider the case of the parabolic obstacle problem \eqref{e:intro:ostacolo:para}. By Proposition \ref{p:ost:loja}, we have that the functional $\mathcal F_{\text{\tiny\sc ob}}$ satisfies the constrained \L ojasiewicz inequality with $\gamma=\sfrac12$, while Lemma \ref{l:cont} implies that the flow of $\mathcal F_{\text{\tiny\sc ob}}$ is continuous with respect to the initial datum. Thus, by Proposition \ref{p:decay} and Corollary \ref{c:1}, we have that the solution $u:[0,+\infty[\,\mapsto L^2(B_1)$ of \eqref{e:intro:ostacolo:para} converges in $L^2(B_1)$ to the unique stationary solution $\varphi$, which is also the unique solution of \eqref{e:intro:ostacolo}. In order to get the convergence in $H^1(B_1)$ and its  rate, we estimate 
\begin{align*}
\int_{B_1}|\nabla (u(t)-\varphi)|^2\,dx
&=\int_{B_1}|\nabla u(t)|^2\,dx-\int_{B_1}|\nabla \varphi|^2\,dx-2\int_{B_1}\nabla \varphi\cdot\nabla (u(t)-\varphi)\,dx\\
&=\int_{B_1}|\nabla u(t)|^2\,dx-\int_{B_1}|\nabla \varphi|^2\,dx+2\int_{B_1} (u(t)-\varphi)\ind_{\{\varphi>0\}}\,dx\\
&=2\big(\mathcal F_{\text{\tiny\sc ob}}(u(t))-\mathcal F_{\text{\tiny\sc ob}}(\varphi)\big)+2|B_1|^{\sfrac12}\|u(t)-\varphi\|_{L^2(B_1)}. 
\end{align*}
Thus, \eqref{e:flow1:convergence} follows by the exponential estimates in Corollary \ref{c:1}.
 
In the case of the thin-obstacle problem \eqref{e:intro:thin:para},  we first apply Proposition \ref{p:sottile} obtaining that the constrained \L ojasiewicz inequality (with $\gamma=\sfrac12$) holds for the functional $\mathcal F_{\text{\tiny\sc th}}$, along the solution $u(t)$ of \eqref{e:intro:thin:para}. On the other hand, Lemma \ref{l:cont} implies that the flow of $\mathcal F_{\text{\tiny\sc th}}$ is continuous with respect to the initial datum. As a consequence, by Corollary \ref{c:1} (and Remark \ref{oss:main:loja}), we have that the solution $u(t)$ converges in $L^2(B_1)$ to the unique solution $\varphi$ of \eqref{e:intro:thin}. Using the notations $H=\{x_d>0\}$ and $B_1^+=H\cap B_1$, we calculate
\begin{align*}
\int_{B_1^+}|\nabla (u(t)-\varphi)|^2\,dx
&=\int_{B_1^+}|\nabla u(t)|^2\,dx-\int_{B_1^+}|\nabla \varphi|^2\,dx+2\int_{\partial H\cap B_1}\frac{\partial \varphi}{\partial x_d} (u(t)-\varphi)\,dx'\\
&\le\int_{B_1^+}|\nabla u(t)|^2\,dx-\int_{B_1^+}|\nabla \varphi|^2\,dx=\mathcal F_{\text{\tiny\sc th}}(u(t))-\mathcal F_{\text{\tiny\sc th}}(\varphi). 
\end{align*}
Thus, the conclusion follows by the estimate \eqref{e:main:decay_rate_1} of Corollary \ref{c:1}.\qed

\subsection{Proof of Theorem \ref{t:flow2}} Let us first treat the case of the parabolic obstacle problem \eqref{e:intro:ostacolo:M} on the sphere. Due to Lemma \ref{l:cont} and Proposition \ref{p:ostacolo:loja}, the hypotheses of Proposition \ref{p:decay} are satisfied. Thus, $u(t)$ converges to a function $u_\infty$, which is a critical point of $\mathcal F_{\text{\tiny\sc ob}}^\lambda$ in $\mathcal K_{\text{\tiny\sc ob}}^\mS$ (in the sense of \eqref{e:main:critical}). Now, since the critical points of $\mathcal F_{\text{\tiny\sc ob}}^\lambda$ in $\mathcal K_{\text{\tiny\sc ob}}^\mS$ (which are in a neighborhood of $\mathcal S_\lambda$) are classified, we get that $u_\infty\in \mathcal S_\lambda$. Moreover, Proposition \ref{p:decay} implies that 
$$\|u(t)-u_\infty\|_{L^2(\mS)}\le Ct^{-\frac{\gamma}{1-2\gamma}}\qquad\text{and}\qquad \mathcal F_{\text{\tiny\sc ob}}^\lambda(u(t))-\mathcal F_{\text{\tiny\sc ob}}^\lambda(u_\infty)\le Ct^{-\frac{1}{1-2\gamma}}.$$ 
Now, reasoning as in the proof of Theorem \ref{t:flow1}, we obtain
\begin{align*}
\int_{\mS}|\nabla (u(t)-u_\infty)|^2
&=\int_{\mS}|\nabla u(t)|^2-\int_{\mS}|\nabla u_\infty|^2-2\int_{\mS}\nabla u_\infty\cdot\nabla (u(t)-u_\infty)\\
&=\int_{B_1}|\nabla u(t)|^2\,dx-\int_{B_1}|\nabla u_\infty|^2\,dx+2\int_{B_1} (u(t)-u_\infty)(1-\lambda u_\infty)\,dx\\
&=2\big(\mathcal F_{\text{\tiny\sc ob}}^\lambda(u(t))-\mathcal F_{\text{\tiny\sc ob}}^\lambda(u_\infty)\big)+\lambda\int_{\mS}|u(t)-u_\infty|^2,
\end{align*}
which finally gives \eqref{e:in_culo_decay2}.

The case of the parabolic thin-obstacle problem is analogous. Using Lemma \ref{l:cont} and Proposition \ref{p:ostacolo:loja} with $\gamma\in\,]0,\sfrac12[\,$, we get by Proposition \ref{p:decay} that $u(t)$ converges to a function $u_\infty$, which is a critical point of $\mathcal F_{\text{\tiny\sc th}}^\lambda$ in $\mathcal K_{\text{\tiny\sc th}}^\mS$. Since the critical points of $\mathcal F_{\text{\tiny\sc th}}^\lambda$ in $\mathcal K_{\text{\tiny\sc th}}^\mS$ are classified, we get that $u_\infty\in\mathcal S_\lambda$. We denote by $\mS^+$ the upper half-sphere $\{x_d>0\}\cap\partial B_1$ and we calculate
\begin{align*}
\int_{\mS^+}|\nabla (u(t)-u_\infty)|^2
&=\int_{\mS^+}|\nabla u(t)|^2-\int_{\mS^+}|\nabla u_\infty|^2-2\lambda\int_{\mS^+}u_\infty (u(t)-u_\infty)\\
&=\mathcal F_{\text{\tiny\sc th}}(u(t))-\mathcal F_{\text{\tiny\sc th}}(u_\infty)+\lambda \int_{\mS^+}|u(t)-u_\infty|^2. 
\end{align*}
Now, using Proposition \eqref{p:decay}, we get \eqref{e:in_culo_decay2}. \qed

\subsection{Proof of Theorem \ref{t:log}} Both the claims (OB) and (TH) follow by Proposition \ref{p:epiK}. Let us check that the conditions (SL), (FL) and (\L S) of Proposition \ref{p:epiK} are satisfied. 

(SL) In order, to prove the slicing condition (SL), we set 
\begin{equation}\label{e:keyGk}
\G(u):=\int_{B_1}|\nabla u|^2 \,dx -k\, \int_{\de B_1} u^2 \,d\HH^{d-1}\,,
\end{equation}
for some $k>0$ and $u=u(r,\theta)$. Setting $\theta\in\partial B_1$, $d\theta=d\HH^{d-1}$ and we calculate
\begin{align}
\mathcal G(r^ku)
&=\int_0^1 \int_{\de B_1}\left(|k\,r^{k-1} u+r^k \de_r u|^2+r^{2k-2}|\nabla_\theta u|^2\right)\,d\theta\,r^{d-1}\,dr-k\,\int_{\de B_1} u^2 \,d\theta\notag\\
&=\int_0^1 \int_{\de B_1} \left( k^2 r^{2k-2}u^2+r^{2k} |\de_r u|^2+ k\, r^{2k-1} \de_r(u^2)+r^{2k-2} |\nabla_\theta u|^2 \right)\,d\theta\,r^{d-1}\,dr-k\,\int_{\de B_1} u^2 \,d\theta\notag\\
&=\int_0^1 \int_{\de B_1} \left( k^2 r^{2k-2}u^2+r^{2k} |\de_r u|^2- k(2k+d-2)\, r^{2k-2} u^2+r^{2k-2} |\nabla_\theta u|^2 \right)\,d\theta\,r^{d-1}\,dr\notag\\
&=\int_0^1 r^{2k+d-3} \int_{\de B_1} \left(|\nabla_\theta u|^2  - k(k+d-2)\,u^2\right)\,d\theta\,dr+\int_0^1 r^{2k+d-1}\int_{\de B_1} |\de_r u|^2\,d\theta\,dr\notag
\end{align}
that is, if we set 
\begin{equation}\label{e:keyFk}
\cF(\phi):=\int_{\de B_1} \left(|\nabla_\theta \phi|^2  - \lambda(k)\,\phi^2\right)\,d\HH^{d-1}\,,\quad\text{where}\quad\lambda(k)=k(k+d-2),
\end{equation}
then we have the slicing equality 
\begin{equation}
\label{e:sl1}
\G(r^ku)=\int_0^1 \cF\big(u(r,\cdot)\big) r^{2k+d-3} \,dr+\int_0^1r^{2k+d-1}\int_{\de B_1} |\de_r u|^2\,d\HH^{d-1}\,dr.
\end{equation}
Now, setting $k=2m$, $\lambda=\lambda(2m)$, $\mathcal G=\mathcal G_{\text{\tiny\sc th}}$ and $\mathcal F=\mathcal F_{\text{\tiny\sc th}}^\lambda$, we get the slicing inequality (SL) for the thin-obstacle problem (at the singular points, where the homogeneity of the blow-up is $2m$). 
Using the identity \eqref{e:sl1} with $k=1$ and $\lambda=2d$, together with a simple change in polar coordinates, we get (SL) for the obstacle problem, with $\mathcal F=\mathcal F_{\text{\tiny\sc ob}}^\lambda$ and $\mathcal G=\mathcal G_{\text{\tiny\sc ob}}$.

(FL) The existence of the flow (FL) follows from the general existence result \cite{brezis}. 
 
(\L S) For what concerns the coinstrained \L ojasiewicz inequality (\L S), in the case of the obstacle problem it is enough to apply Proposition \ref{p:ostacolo:loja} with $\M=\mathbb S^{d-1}$ and $\lambda = 2d$, while in the case of the thin-obstacle it follows from Proposition \ref{prop:thin:loja}. \qed

\end{proof}

\bigskip\bigskip
\noindent {\bf Acknowledgments.} 
The second author has been partially supported by the NSF grant DMS 1810645. The third author has been partially supported by Agence Nationale de la Recherche (ANR) by the projects GeoSpec (LabEx PERSYVAL-Lab, ANR-11-LABX-0025-01) and CoMeDiC (ANR-15-CE40-0006). 

\bibliographystyle{plain}
\bibliography{references-Cal}

\end{document}